\documentclass[reqno]{amsart}
\usepackage{epstopdf,epsfig}
\usepackage{graphicx}
\usepackage{subcaption} 
\usepackage{tikz}
\usepackage{makecell}
\usepackage{empheq}
\usepackage{url}
\usepackage{hyperref}

\newtheorem{theorem}{Theorem}[section]
\newtheorem{lemma}[theorem]{Lemma}

\theoremstyle{definition}

\theoremstyle{remark}
\newtheorem{remark}[theorem]{Remark}

\numberwithin{equation}{section}

\newcommand{\p}{\partial}

\def\d{\ensuremath{\mathrm{d}}}

\newcommand{\pd}[2]{\dfrac{\partial #1}{\partial #2}}

\newcommand{\norm}[1]{\left\|#1\right\|}

\newcommand{\xx}{{\mathbf{x}}}
\newcommand{\yy}{{\mathbf{y}}}
\newcommand{\zz}{{\mathbf{z}}}

\newcommand{\nn}{{\mathbf{n}}}

\newcommand{\RR}{{\mathbb{R}}}

\newcommand{\into}{{\int_\Omega}}
\newcommand{\intpo}{{\int_{\partial\Omega}}}
\newcommand{\Rd}{{R_{\delta}(\xx,\yy)}}
\newcommand{\bRd}{{\bar{R}_{\delta}(\xx,\yy)}}

\newenvironment{equationa*}{\begin{equation*}\begin{aligned}} {\end{aligned}\end{equation*}}

\begin{document}

\title[]{Asymptotically Compatible Error Bound of Finite Element Method for Nonlocal Diffusion Model with An Efficient Implementation}

\author[]{Yanzun Meng}
\address{Yanzun Meng: Department of Mathematical Sciences, Tsinghua University,
Beijing, China, 100084.}
\curraddr{}
\email{myz21@mails.tsinghua.edu.cn}
\thanks{}

\author[]{Zuoqiang Shi}
\address{Zuoqiang Shi: Yau Mathematical Sciences Center, Tsinghua University, 
Beijing, China, 100084. \&
Yanqi Lake Beijing Institute of Mathematical Sciences and Applications,
 Beijing, China, 101408.}
\email{zqshi@tsinghua.edu.cn}
\thanks{This work is supported by National Natural Science Foundation of China (NSFC) 92370125.}





\begin{abstract}
 This paper presents an asymptotically compatible error bound for the finite element method (FEM) applied to a nonlocal diffusion model. 
 The analysis covers two scenarios: meshes with and without shape regularity. For shape-regular meshes, the error is bounded by \(O(h^k + \delta)\), where \(h\) is the mesh size, \(\delta\) is the nonlocal horizon, and \(k\) is the order of the FEM basis. 
 Without shape regularity, the bound becomes \(O(h^{k+1}/\delta + \delta)\).  
 In addition, we present an efficient implementation of the finite element method of nonlocal model. 
 The direct implementation of the finite element method of nonlocal model requires computation of $2n$-dimensional integrals which are very expensive. 
 For the nonlocal model with Gaussian kernel function, we can decouple the $2n$-dimensional integral to 2-dimensional integrals which reduce the computational cost tremendously. 
 Numerical experiments verify the theoretical results and demonstrate the outstanding performance of the proposed numerical approach. 
\end{abstract}

\maketitle



\section{Introduction}
\label{Sec:Introduction}
Nonlocal modeling has emerged as a powerful framework in recent decades, offering advantages over traditional differential operator-based approaches, particularly for problems involving singularities or anomalous behavior. By replacing differential operators with integral operators, nonlocal models can capture complex phenomena that classical partial differential equations (PDEs) struggle to describe. 
Nonlocal models have found applications in diverse fields, including
anomalous diffusion~\cite{andreu2010nonlocal,bucur2016nonlocal,vazquez2012nonlinear,burch2011classical}, fracture mechanics in peridynamics~\cite{askari2008peridynamics,oterkus2012peridynamic,silling2010crack,du2011mathematical,silling2010peridynamic}, traffic flow~\cite{QiangDu2023DiscreteandContinuousDynamicalSystems}, imaging process~\cite{osher2017low} and semi-supervised learning~\cite{shi2017weighted,wang2018non,tao2018nonlocal}. 
Given their broad applicability, the development of efficient and accurate numerical methods for nonlocal models has attracted significant attention.

To solve the nonlocal models, many numerical methods have been proposed in the literature, include difference method \cite{doi:10.1137/13091631X}, finite element method \cite{chen2011continuous,du2013posteriori,du2013convergent}, spectral method \cite{du2017fast, doi:10.1137/15M1039857, DU2017118}, 
collocation method \cite{zhang2016nodal, zhang2018accurate} and mesh free method \cite{bond2015galerkin,silling2005meshfree,lehoucq2018meshless,lehoucq2016radial}. Among the various numerical approaches, the finite element method (FEM) stands out due to its flexibility and robustness. In this paper, we focus the finite element discretization of a nonlocal diffusion model
\begin{align}
  &\frac{1}{\delta^2}\into \Rd (u(\xx)-u(\yy))\d \yy +\into \bRd u(\yy)\d \yy\notag\\
  &\hspace{4cm}=\into \bRd f(\yy)\d \yy+2\intpo \bRd g(\yy)\d S_\yy.\label{eq:nonlocal model intro}
\end{align}
where $\Rd$ and $\bRd$ are integral kernels, which are typically chosen as radially symmetric and limited to a spherical neighborhood of radius $2\delta$. $f$ and $g$ are given functions. The details of the above nonlocal model are given in Section \ref{sec:Nonlocal diffusion model}. It has been proved that under some mild assumptions, the solution of above nonlocal model converges to the solution of the following elliptic equation with Neumann boundary condition  
\begin{equation}
  \label{eq:local model intro}
  \left\{
    \begin{aligned}
      -&\Delta u(\xx)+u(\xx)=f(\xx),\quad &\xx \in\Omega,\\
      &\pd{u}{\nn}(\xx)=g(\xx),&\xx\in \partial\Omega,
    \end{aligned}
  \right.
\end{equation}
as $\delta$ goes to zero \cite{shi2017convergence}.  

In the theoretical part of this paper, we analyze the error between the finite element solution of the nonlocal model \eqref{eq:nonlocal model intro} and the exact solution of the local model \eqref{eq:local model intro}, denoted as $u_h-u$.
If the shape regularity is preserved as mesh size $h\rightarrow 0$, we prove that the error is $O(h^k+\delta)$ in $L^2$ norm with $k$-th order finite element basis. 
For $H^1$ norm, due to the absence of $H^1$ coercivity for nonlocal diffusion model, we can not get the bound of $\|u_h-u\|_{H^1(\Omega)}$ directly. However, we introduce a gradient recovery method such that the error gradient also has the bound of $O(h^k+\delta)$ after recovery. This theoretical result shows that the finite element solution of the nonlocal model converges to the solution of the local model as $h, \delta$ go to zero without any requirement on the relation between $h$ and $\delta$. This property is very important to guarantee that the finite element method is asymptotically compatible (AC) as introduced by Du and Tian \cite{tian2014asymptotically}. In \cite{tian2014asymptotically}, a theoretical framework of AC scheme was established to show that under some general assumptions, the Galerkin finite element approximation is always asymptotically compatible as long as the continuous piecewise linear functions are included in the finite element space. For a specific nonlocal diffusion model \eqref{eq:nonlocal model intro}, we get the optimal $H^1$ convergence rate in $h$ after introducing a gradient recovery strategy. The convergence rate in $\delta$ is first order which is also optimal in the sense that the convergence rate of the nonlocal model itself is also first order.  

If the shape regularity is not preserving when mesh size $h$ goes to zero, the error bound becomes $O(h^{k+1}/\delta+\delta)$. In this case, the finite element method is asymptotically compatible with condition $h^{k+1}/\delta\rightarrow 0$. This is a reasonable result, since the finite element method is not convergent for the local problem without shape regularity. 

Although the finite element method for nonlocal model has good theoretical properties, the implementation of the nonlocal finite element method is very challenging. The most difficult part lies in the assembling of the stiffness matrix. In this process, we need to compute following integral many times. 
\begin{align*}
  \left<\mathcal{L}_\delta \psi_i, \psi_j\right> = \int_{\Omega}\int_{B(\xx,2\delta)\cap \Omega}\gamma_\delta(\xx,\yy)(\psi_i(\xx)-\psi_i(\yy))\psi_j(\xx)\d\yy\d\xx,
\end{align*}
Where $\psi_i,\psi_j$ are the node basis functions.
If $\Omega$ is a domain in $\RR^n$, the above integral is in fact a $2n$-dimensional integral. Assembling the stiff matrix requires calculating this kind of integral for numerous times, which brings expensive computation cost. 
Meanwhile, the kernel $\gamma_\delta$ is usually nearly-singular, and dealing with the intersection of the Euclidean ball $B(\xx,\delta)$ and the mesh is also challenging.  Despite considerable efforts have been made to mitigate these issues, 
such as \cite{zhang2016quadrature, PASETTO2022115104} designed efficient quadrature method and \cite{doi:10.1142/S0218202521500317} polygonally approximated the Euclidean ball,  
the implementation of nonlocal finite element is still a challenging task.

For Gaussian kernel and tensor-product domain, we propose a fast implementation of the nonlocal finite element method. In this case, the $2n$-dimensional integral can be separated to the product of 2d integrals, which reduces the computational cost tremendously. For the domain which can be decomposed to the union of tensor-product domains, the method is still applicable.  

The rest of this paper is organized as follows. In Section \ref{sec:Nonlocal diffusion model and conformal finite element discretization}, we give the formulation of nonlocal diffusion model and introduce the finite element discretization.
The details of the error analysis are presented in Section \ref{sec:Error analysis of finite element method}. 
Subsequently, the fast implementation is introduced in Section \ref{sec: Fast Implementation} and numerical experiments are demonstrated in Section \ref{sec: Numerical Experiments}.
 
\section{Nonlocal finite element discretization and main results}
\label{sec:Nonlocal diffusion model and conformal finite element discretization}
This section will introduce the configuration of our nonlocal diffusion model with its local counterpart. 
To solve this nonlocal problem, a conformal finite element discretization is designed. 
The error estimations between the finite element solution and the PDE solution will be stated in this section. Additionally, we also design a method to approximate the gradient of the local solution.

\subsection{Nonlocal diffusion model}
\label{sec:Nonlocal diffusion model}
In this paper, we consider the following partial differential equation with Neumann boundary.
\begin{equation}
  \label{eq:local model}
  \left\{
    \begin{aligned}
      -&\Delta u(\xx)+u(\xx)=f(\xx),\quad &\xx \in\Omega,\\
      &\pd{u}{\nn}(\xx)=g(\xx),&\xx\in \partial\Omega,
    \end{aligned}
  \right.
\end{equation}
where $\Omega\subset\RR^n$ is a bounded and connected domain.
The nonlocal counterpart of this equation is given as follows
\begin{align}
  &\frac{1}{\delta^2}\into \Rd (u(\xx)-u(\yy))\d \yy +\into \bRd u(\yy)\d \yy\notag\\
  &\hspace{4cm}=\into \bRd f(\yy)\d \yy+2\intpo \bRd g(\yy)\d S_\yy.\label{eq:nonlocal model}
\end{align}
The kernel functions $R_\delta$ and $\bar{R}_\delta$ in (\ref{eq:nonlocal model}) 
are derived from a function $R$ which satisfies the following conditions:
\begin{itemize}
  \item[(a)] (regularity) $R\in C^1([0,+\infty))$;
  \item[(b)] (positivity and compact support)
  $R(r)\ge 0$ and $R(r) = 0$ for $\forall r >1$;
  \item[(c)] (nondegeneracy)
   $\exists \gamma_0>0$ so that $R(r)\ge\gamma_0$ for $0\le r\le\frac{1}{2}$. 
\end{itemize}
With this function, we can further define 
\begin{align*}
  \bar{R}(r)=\int_r^{+\infty} R(s)\d s.
\end{align*}
We can find $\bar{R}$ also satisfies the above three conditions. 
With these two univariate functions, we can get the corresponding kernel function with scaling transformation as follows 
\begin{align}
  \label{eq:kernel def}
  R_\delta(\xx,\yy)=\alpha_n \delta^{-n}R\left(\frac{|\xx-\yy|^2}{4\delta^2}\right),\quad \bar{R}_\delta(\xx,\yy)=\alpha_n \delta^{-n}\bar{R}\left(\frac{|\xx-\yy|^2}{4\delta^2}\right).
\end{align}
Here $\alpha_n$ is a normalization constant such that 
\begin{align*}
    \int_{\RR^n}\alpha_n \delta^{-n}\bar{R}\left(\frac{|\xx-\yy|^2}{4\delta^2}\right)\d \yy =\alpha_n S_n\int_0^2 \bar{R}(r^2/4)r^{n-1}\d r=1,
\end{align*}

With the configuration as above, we can illustrate our finite element scheme.

\subsection{Finite element discretization.}
We next consider solving the nonlocal model (\ref{eq:nonlocal model}) with finite element method. Let $\Omega_h$ be a polyhedral approximation of $\Omega$, and $\mathcal{T}_h$ be the mesh associated with $\Omega_h$,
where $h = \max_{T\in\mathcal{T}_h}\mbox{diam}(T)$ is the maximum diameter. 
Additionally, the radius of the inscribed ball of $T$ is denoted as $\rho(T)$ and $\rho = \min_{T\in\mathcal{T}_h}\rho(T)$.
We focus on the continuous $k$-th order finite element space defined on $\Omega_h$, i.e.
\begin{equation}
	S_h = \left\{v_h \in C^0(\Omega_h): v_h|_T \in \mathbb{P}_k(T), \quad  \forall T \in \Omega_h \right\}.
\end{equation}
If $\mathcal{T}_h$ is a simplicial mesh, such as triangular mesh in 2D and tetrahedral mesh in 3D, $\mathbb{P}_k$ denotes the set of all $k$-th order polynomials in $T$. 
Meanwhile, for Cartesian mesh, e.g. rectangular mesh in 2D and cuboidal mesh in 3D, $\mathbb{P}_k$ will be chosen as $k$-th tensor-product polynomial space.

The finite element discretization of the nonlocal diffusion model is to find $u_h\in S_h$ such that 
\begin{align}
  \label{eq:fem}
  \left<L_\delta u_h, v_h\right>_{\Omega_h}=\left<\bar{f}_\delta,v_h\right>_{\Omega_h},\quad \forall v_h\in S_h,
\end{align}
with 
$\bar{f}_\delta(\xx)=\int_\Omega \bar{R}_\delta(\xx,\yy)f(\yy)\d \yy+2\int_{\p \Omega} \bar{R}_\delta(\xx,\yy)g(\yy)\d S_{\yy}$ and 
\begin{equation}
\label{eq:L_delta}
L_\delta v(\xx) = \frac{1}{\delta^2}\int_\Omega R_\delta(\xx,\yy)(v(\xx)-v(\yy))\d \yy+\int_\Omega \bar{R}_\delta(\xx,\yy)v(\yy)\d \yy,\quad \forall v\in L^2(\Omega).
\end{equation}
The binary operator $<\cdot,\cdot>_{\Omega_h}$ in (\ref{eq:fem}) denotes the inner product in $\Omega_h$, i.e. 
\begin{align*}
  \left<u, v\right>_{\Omega_h}=\int_{\Omega_h} u(\xx)v(\xx)\d \xx.
\end{align*}
For the sake of simplification, we focus on the case $\Omega=\Omega_h$ which means that we do not consider the error from domain approximation. 
In the rest of the paper, $\Omega$ and $\Omega_h$ will not be distinguished.

\subsection{Main results.}
We will give the main results of this paper in advance here. The proof of these results can be found in the following sections.
Our results include two key points. Firstly, the $L^2$ error between the nonlocal finite element solution $u_h$ and the solution of the local counterpart $u$ can get an estimation. 
Secondly, based on the solution $u_h$, we can also approximate $\nabla u$.
\begin{theorem}
  \label{th:main}
  Let $u\in H^{\max\{k+1,3\}}(\Omega)$ solve the local model (\ref{eq:local model}) and $u_h$ be the solution of (\ref{eq:fem}). 
  We can obtain 
  \begin{align}
    \norm{u-u_h}_{L^2(\Omega)}\leq C\left(\frac{h^{k+1}}{\max\{\rho,\delta\}}+\delta\right)\norm{u}_{H^{\max\{k+1, 3\}}(\Omega)},\label{eq:L2error result}
  \end{align}
  where $\rho$ is the minimal radius of the inscribed ball of the elements and $C$ is a constant independent of $\delta$ and $h$.
\end{theorem}
\begin{remark}
  \label{remark1}
  Noticing the result (\ref{eq:L2error result}) indicates the following result
  \begin{align}
    \norm{u-u_h}_{L^2(\Omega)}\leq C\left(h^k + \delta\right)\norm{u}_{H^{\max\{k+1, 3\}}(\Omega)},\label{eq:shape regular result}
  \end{align}
  under the shape regular condition, i.e. $\frac{h}{\rho}$ is bounded. More importantly, this is an asymptotically compatible result. 
  In other words, as long as our mesh is shape regular, the finite element solution converges to the local solution as $\delta\rightarrow 0$ and $h\rightarrow 0$ independently. For irregular mesh, this result also indicates the following error bound depending only on $\delta$ and $h$
  \begin{align}
    \norm{u-u_h}_{L^2(\Omega)}\leq C\left(\frac{h^{k+1}}{\delta} + \delta\right)\norm{u}_{H^{\max\{k+1, 3\}}(\Omega)}.
  \end{align}
\end{remark}

Moreover, in this paper, we also design a method to approximate the gradient of the local solution. For $v\in L^2(\Omega)$, we define
\begin{align}
  S_\delta v(\xx)=\frac{1}{w_\delta(\xx)}\into \Rd v(\yy)\d \yy,\quad w_\delta(\xx)=\int_\Omega \Rd\d \yy.\label{eq:Sdelta}
\end{align} 
Then we can obtain the following theorem.
\begin{theorem}
  \label{th:gradient}
  Let $u\in H^{\max\{k+1,3\}}(\Omega)$ solve the local model (\ref{eq:local model}) and $u_h$ be the solution of (\ref{eq:fem}). With the correction term 
  \begin{align}
    \mathbf{F}_\delta(\xx) = \frac{1}{w^2_\delta(\xx)}\int_{\partial\Omega}\int_\Omega\Rd R_\delta(\xx,\zz)g(\zz)((\yy-\zz)\cdot \nn(\zz))\nn(\zz)\d S_\zz\d \yy,\label{eq:correction}
  \end{align}
  we can get 
  \begin{align}
    \norm{\nabla u - \left(\nabla S_\delta u_h - \mathbf{F}_\delta\right)}_{L^2(\Omega)}^2\leq C\left(\frac{h^{k+1}}{\max\{\rho,\delta\}} + \delta\right)\norm{u}_{H^{\max\{k+1,3\}}(\Omega)},\label{correction estimation}
  \end{align}
  where $g(\zz)$ is the Neumann boundary term in (\ref{eq:local model}) and $C$ is a constant independent of $\delta$ and $h$. 
\end{theorem}

Similar to Remark \ref{remark1}, with shape regular condition, above theorem can also get an asymptotically compatible version. For a more important point, we give the following remark.
\begin{remark}
  \label{remark:gradient approximation}
  The complicated correction term (\ref{eq:correction}) is introduced for dealing with the loss of half an order of convergence in terms of $\delta$. In other words, without the correction term, the result will become
  \begin{align}
    \norm{\nabla u - \nabla S_\delta u_h}_{L^2(\Omega)}^2\leq C\left(\frac{h^{k+1}}{\max\{\rho,\delta\}} + \sqrt{\delta}\right)\norm{u}_{H^{\max\{k+1,3\}}(\Omega)}.\label{eq:gradient approximation1}
  \end{align}
  In fact, we will find in the subsequent sections, this relatively low order is caused by the error between $\nabla u$ and $\nabla S_\delta u$ in $\Omega_{2\delta}$, where $ \Omega_{2\delta} = \left\{\xx\big|d(\xx,\partial\Omega)\leq 2\delta\right\}$.
  This means even if $u_h$ exactly equals to $u$, the error with respect to $\delta$ in this narrow band-region is only of half order. Without considering $\Omega_{2\delta}$, the error estimation becomes
  \begin{align}
    \norm{\nabla u - \nabla S_\delta u_h}_{L^2(\Omega\backslash\Omega_{2\delta})}^2\leq C\left(\frac{h^{k+1}}{\max\{\rho,\delta\}} + \delta\right)\norm{u}_{H^{\max\{k+1,3\}}(\Omega)}.\label{eq:gradient approximation2}
  \end{align}
\end{remark}

\section{Error analysis of finite element method}
\label{sec:Error analysis of finite element method}
The proof of the error estimations in Section \ref{sec:Nonlocal diffusion model and conformal finite element discretization} will be present in this section. We start from some technical results. Then both Theorem \ref{th:main} and Theorem \ref{th:gradient} can be derived based on these results.

\subsection{Technical results.}
In order to analyze our nonlocal finite element scheme, we should introduce the following nonlocal energy at first.
\begin{align}
  (E_\delta(v))^2
  &=\frac{1}{2\delta^2}\into\into \Rd (u(\xx)-u(\yy))^2\d\xx\d\yy+\into\into\bRd u(\xx)u(\yy)\d\xx\d\yy.\label{eq:nonlocal energy}
\end{align}
It is easy to verify $(E_\delta(v))^2$ is actually the inner product of $L_\delta v$ and $v$.
For $E_\delta(v)$, we have some technical results.
\begin{lemma}
  \label{lemma:three results}
  There exist constants $C$ independent of $\delta$ such that for $v\in L^2(\Omega)$ along with $E_\delta(v)$ and $S_\delta v$ defined in (\ref{eq:nonlocal energy})(\ref{eq:Sdelta}),  
  \begin{align}
    &E_\delta(v)\leq \frac{C}{\delta}\norm{v}_{L^2(\Omega)}\label{eq: ineq1}\\
    &\norm{\nabla (S_\delta v)}_{L^2(\Omega)}\leq CE_\delta(v)\label{eq: ineq2}\\
    &\norm{v}_{L^2(\Omega)}\leq CE_\delta(v).\label{eq: ineq3}
  \end{align}
\end{lemma}
\begin{proof}
  We firstly prove estimation (\ref{eq: ineq1}). For the second term of $(E_\delta(v))^2$,
  \begin{align*}
    \into\into \bRd v(\xx)v(\yy)\d\xx\d\yy&\leq \frac{1}{2}\into\into\bRd (v^2(\xx)+v^2(\yy))\d\xx\d\yy\\
    &= \into v^2(\xx)\into \bRd \d\yy\d\xx\\
    &\leq C\norm{v}_{L^2(\Omega)}^2.
  \end{align*}
  As for the first term of $(E_\delta(v))^2$,
  \begin{align*}
    \frac{1}{2\delta^2}\into\into \Rd (u(\xx)-u(\yy))^2\d\xx\d\yy&=\frac{1}{\delta^2}\into\into \Rd (v^2(\xx)+v(\xx)v(\yy))\d\xx\d\yy\\
    &\leq \frac{2}{\delta^2}\into v^2(\xx)\into \Rd \d\yy\d\xx\\
    &\leq \frac{C}{\delta^2}\norm{v}_{L^2(\Omega)}^2.
  \end{align*}
  Here we have proved (\ref{eq: ineq1}).

  For the second result (\ref{eq: ineq2}), \cite{shi2017convergence} provides an inequality
  \begin{align}
    \norm{\nabla (S_\delta v)}^2_{L^2(\Omega)}\leq \frac{C}{2\delta^2}\into\into \Rd (v(\xx)-v(\yy))^2\d\xx\d\yy.\label{eq:nabla S}
  \end{align}
  We just need to show the first term of $(E_\delta(v))^2$ can be bounded by $(E_\delta(v))^2$, i.e.
  \begin{align}
    \label{eq: control}
    \frac{1}{2\delta^2}\into\into\Rd (v(\xx)-v(\yy))^2\d\xx\d\yy\leq C(E_\delta(v))^2.
  \end{align}
  In fact, we can get (\ref{eq: control}) with the following estimation.
  \begin{align}
    &\into\into \bRd v(\xx)v(\yy)\d\xx\d\yy\notag\\
    \geq&-\frac{1}{4}\into\into\bRd (v(\xx)-v(\yy))^2\d\xx\d\yy\notag\\
    \geq&-C\into\into\Rd (v(\xx)-v(\yy))^2\d\xx\d\yy.\label{eq:two kernel}
  \end{align}
  The last inequality above can be found in \cite{wang2023nonlocal}. With these estimations, we can conclude (\ref{eq: ineq2}).

  We lastly turn to the proof of (\ref{eq: ineq3}). By reusing the last inequality in (\ref{eq:two kernel}) and denoting $\bar{w}_\delta(\xx)=\into \bRd\d\yy$,
  we get 
  \begin{align*}
    \norm{v}_{L^2(\Omega)}^2&=\into v^2(\xx)\frac{1}{\bar{w}_\delta(\xx)}\into \bRd\d\yy\d\xx\\
    &\leq C\into\into \bRd v^2(\xx)\d\xx\d\yy\\
    &=\frac{C}{2}\into\into \bRd v^2(\xx)\d\xx\d\yy+\frac{C}{2}\into\into \bRd v^2(\yy)\d\xx\d\yy\\
    &=\frac{C}{2}\into\into \bRd (v(\xx)-v(\yy))^2\d\xx\d\yy+C\into\into \bRd v(\xx)v(\yy)\d\xx\d\yy\\
    &\leq C\into\into \Rd (v(\xx)-v(\yy))^2\d\xx\d\yy+C\into\into \bRd v(\xx)v(\yy)\d\xx\d\yy.
  \end{align*}
  Here we have finished the proof of Lemma \ref{lemma:three results}.
\end{proof}

Moreover, with the help of (\ref{eq: control})(\ref{eq: ineq3}), we can prove $E_\delta$ is weakly subadditive, i.e. 
\begin{lemma}
  \label{lemma: subadditive}
  There exists $C$ independent of $\delta$ such that for $v,w\in L^2(\Omega)$, 
  \begin{align*}
    E_\delta(v+w)\leq C(E_\delta(v)+E_\delta(w)).
  \end{align*}
\end{lemma}
\begin{proof}
  We can find 
  \begin{align*}
    &(E_\delta(v+w))^2\\
    =&\frac{1}{2\delta^2}\into\into \Rd((v(\xx)+w(\xx))-(v(\yy)+w(\yy)))^2\\
    &\hspace{1cm}+\into\into \bRd (v(\xx)+w(\xx))(v(\yy)+w(\yy))\d\xx\d\yy\\
    \leq& \frac{C}{2\delta^2}\into\into \Rd\left((v(\xx)-v(\yy))^2+(w(\xx)-w(\yy))^2\right)\d\xx\d\yy\\
    &\hspace{1cm}+\into\into \bRd (v(\xx)v(\yy)+w(\xx)w(\yy)+v(\xx)w(\yy)+w(\xx)v(\yy))\d\xx\d\yy\\
    \leq&C(E_\delta(v))^2+C(E_\delta(w))^2+C\into\into \bRd (v^2(\xx)+w^2(\xx)+v^2(\yy)+w^2(\yy))\d\xx\d\yy\\
    \leq&C(E_\delta(v))^2+C(E_\delta(w))^2+C\norm{v}_{L^2(\Omega)}^2+C\norm{w}_{L^2(\Omega)}^2\\
    \leq&C(E_\delta(v))^2+C(E_\delta(w))^2,
  \end{align*}
  which implies the result we need.
\end{proof}

Besides the estimations about $E_\delta(v)$ itself, there are also some results concerning $E_\delta(v)$ and $L_\delta v$, which will be used in the subsequent analysis.
\begin{lemma}
  \label{lemma: technical results-2}
  For $v,w\in L^2(\Omega)$ and $E_\delta(v)$ defined as in (\ref{eq:nonlocal energy}), we have 
  \begin{align}
    &\norm{L_\delta v}_{L^2(\Omega)}\leq \frac{C}{\delta}E_\delta(v),\label{eq:ineq4}\\
    &|\left<L_\delta v,w\right>_\Omega|\leq CE_\delta(v)\norm{w}_{H^1(\Omega)},\label{eq:ineq5}
  \end{align}
  where $C$ is independent of $\delta$.
\end{lemma}

The inequality (\ref{eq:ineq4}) is easy to verify. In fact, 
\begin{align*}
  \norm{L_\delta v}^2_{L^2(\Omega)}&\leq \frac{C}{\delta^4}\into \left|\into \Rd (v(\xx)-v(\yy))\d\yy\right|^2\d\xx+C\into\left|\into \bRd v(\yy)\d\yy\right|^2\d\xx\\
  &\leq \frac{C}{\delta^4}\into \into \Rd (v(\xx)-v(\yy))^2\d\xx\d\yy+C\into\into \bRd v^2(\yy)\d\yy\d\xx\\
  &\leq \frac{C}{\delta^2}\frac{1}{\delta^2}\into \into \Rd (v(\xx)-v(\yy))^2\d\xx\d\yy+C\norm{v}_{L^2(\Omega)}^2\\
  &\leq \frac{C}{\delta^2}(E_\delta(v))^2.
\end{align*}
In the last inequality above, (\ref{eq: ineq3}) and (\ref{eq: control}) are used.

To prove the second result in Lemma  \ref{lemma: technical results-2}, the following estimation is in need.
\begin{lemma}
  \label{lemma:Taylor1}
  There exists a constant $C$ depending only on $\Omega$, such that for $v\in L^2(\Omega)$,
  \begin{align*}
    \into\into \Rd (v(\xx)-v(\yy))^2\d\xx\d\yy\leq C\delta^2\norm{v}_{H^1(\Omega)}^2.
  \end{align*}
\end{lemma}
The proof of Lemma \ref{lemma:Taylor1} can be found in \cite{shi2017convergence}.
With this estimation, we can derive (\ref{eq:ineq5}) as follows.
\begin{align*}
  &|\left<L_\delta v,w\right>_\Omega|\\
  \leq& \frac{1}{2\delta^2}\left|\into\into \Rd (v(\xx)-v(\yy))(w(\xx)-w(\yy))\d\yy\d\xx\right|+\left|\into\into\bRd v(\yy)w(\xx)\d\xx\d\yy\right|\\
  \leq& \frac{1}{2\delta^2}\left(\into\into \Rd (v(\xx)-v(\yy))^2\d\xx\d\yy\right)^{\frac{1}{2}}\left(\into\into \Rd (w(\xx)-w(\yy))^2\d\xx\d\yy\right)^{\frac{1}{2}}\\
  &\qquad +\left(\into\into \bRd v^2(\yy)\d\yy\d\xx\right)^{\frac{1}{2}}\left(\into\into \bRd w^2(\xx)\d\yy\d\xx\right)^{\frac{1}{2}}\\
  \leq &C\left(\frac{1}{2\delta^2}\into\into \Rd (v(\xx)-v(\yy))^2\d\xx\d\yy\right)^{\frac{1}{2}}\left(\frac{1}{2\delta^2}\into\into \Rd (w(\xx)-w(\yy))^2\d\xx\d\yy\right)^{\frac{1}{2}}\\
  &\qquad + \norm{v}_{L^2(\Omega)}\norm{w}_{L^2(\Omega)}\\
  \leq& CE_\delta(v)\norm{w}_{H^1(\Omega)}.
\end{align*}
Here (\ref{eq: ineq3})(\ref{eq: control}) and Lemma \ref{lemma:Taylor1} are applied to get the last inequality above.

\subsection{Error analysis.}
We next start to analyze our nonlocal finite element method. Let $u$ solve the local model (\ref{eq:local model}) and $u_h$ be the solution of (\ref{eq:fem}). If $e_h=u-u_h$, we can find 
\begin{align*}
  \left<L_\delta e_h, v_h\right>_\Omega=\left<r,v_h\right>,\quad \forall v_h\in S_h.
\end{align*}
Here 
\begin{align*}
  r(\xx)=\frac{1}{\delta^2}\into \Rd (u(\xx)-u(\yy))\d\yy+\into \bRd \Delta u(\yy)\d \yy-2\intpo \bRd \pd{u}{\nn}(\yy)\d S_\yy.
\end{align*}
For this truncation error, we have the following lemma to decompose $r(\xx)$ into an interior error and a boundary error.
\begin{lemma}
  \label{lemma:truncation error}
  For arbitrary $u\in H^3(\Omega)$. We denote 
  \begin{align}
    \label{boundary error}
    r_{bd}(\xx)=\sum_{j=1}^n\intpo n^j(\yy)(\xx-\yy)\cdot \nabla(\nabla^j u(\yy))\bRd \d S_\yy
  \end{align}
  and 
  \begin{align*}
    r_{in}(\xx)=r(\xx)-r_{bd}(\xx),
  \end{align*}
  where $n^j(\yy)$ is the $j$-th component of the unit outward normal $\nn(\yy)$ at $\yy\in \partial\Omega$. Then there exist constants $C$ depending only on $\Omega$, such that  
  \begin{align*}
    \norm{r_{in}}_{L^2(\Omega)}\leq C\delta \norm{u}_{H^3(\Omega)}, \quad \norm{r_{bd}}_{L^2(\Omega)}\leq C\delta^{1/2}\norm{u}_{H^3(\Omega)}.
  \end{align*}
\end{lemma}
The proof of this lemma can be found in \cite{shi2017convergence}.
With the notations in Lemma \ref{lemma:truncation error}, we can get $e_h$ satisfies 
\begin{align}
  \left<L_\delta e_h, v_h\right>_\Omega=\left<r_{in}+r_{bd},v_h\right>,\quad \forall v_h\in S_h.\label{eq:e_h}
\end{align}

Additionally, we have another estimation to control the inner product of $r_{bd}$ and a $v\in L^2(\Omega)$ with $E_\delta(v)$.
\begin{lemma}
  \label{lemma:boundary error}
  Let $u\in H^3(\Omega)$ and $r_{bd}$ is defined as (\ref{boundary error}), then there exists a constant $C$ depending only on $\Omega$, for any $v\in S_h$,
  \begin{align*}
    \left|\into v(\xx)r_{bd}(\xx)\d\xx\right|\leq C\delta\norm{u}_{H^3(\Omega)}E_\delta(v).
  \end{align*}
\end{lemma}
\begin{proof}
  \begin{align*}
    \left|\into v(\xx)r_{bd}(\xx)\d\xx\right|&=\left|\sum_{j=1}^n\into v(\xx)\intpo n^j(\yy)(\xx-\yy)\cdot \nabla (\nabla^j u(\yy))\bar{R}_\delta(\xx,\yy)\d S_\yy\d\xx\right|\\
    &\leq C\delta\intpo \norm{H(u)(\yy)}\into\bRd |v(\xx)|\d\xx\d S_\yy\\
    &\leq C\delta\norm{u}_{H^2(\partial\Omega)}\norm{\bar{S}_\delta(|v|)}_{L^2(\partial\Omega)}\\
    &\leq C\delta\norm{u}_{H^3(\Omega)}\norm{\bar{S}_\delta(|v|)}_{H^1(\Omega)},
  \end{align*}
  where $H(u)$ denotes the Hessian of $u$, and 
  \begin{align*}
    \bar{S}_\delta(v)(\xx)=\frac{1}{\bar{w}_\delta(\xx)}\into \bRd v(\yy)\d\yy,\quad \bar{w}_\delta(\xx)=\into \bRd \d\yy.
  \end{align*}
  Moreover, with (\ref{eq:two kernel}), it is easy to verify the results in Lemma \ref{lemma:three results} are also applicable to $\bar{S}_\delta(v)$. Hence, we can get
  \begin{align*}
    &\norm{\bar{S}_\delta(|v|)}_{H^1(\Omega)}^2\\
    \leq& \norm{\bar{S}_\delta(|v|)}_{L^2(\Omega)}^2+C(E_\delta(|v|))^2\\
    \leq& C\into\left(\into \bRd v(\yy)\d\yy\right)^2\d\xx+C(E_\delta(|v|))^2\\
    \leq& C\norm{v}_{L^2(\Omega)}^2+\frac{C}{\delta^2}\into\into \Rd(|v(\xx)|-|v(\yy)|)^2\d\xx\d\yy+C\into\into \bRd |v(\xx)v(\yy)|\d\xx\d\yy\\
    \leq& \frac{C}{\delta^2}\into\into \Rd(v(\xx)-v(\yy))^2\d\xx\d\yy+C\norm{v}_{L^2(\Omega)}^2\\
    \leq& C(E_\delta(v))^2.
  \end{align*}
  Here (\ref{eq: ineq3}) and (\ref{eq: control}) are used again in the last line. 
\end{proof}

\subsection{Proof of Theorem \ref{th:main}.}
With all those preparations above, we can now prove (\ref{eq:L2error result}).
Let $I_h$ denote the projection operator onto $S_h$. Following (\ref{eq:e_h}), we can get 
\begin{align}
  \left<L_\delta e_h,e_h\right>_\Omega&=\left<L_\delta e_h, u-I_h u\right>_\Omega+\left<L_\delta e_h, I_h u-u_h\right>\notag\\
  &=\left<L_\delta e_h, u-I_h u\right>_\Omega+\left<r_{in}+r_{bd}, I_h u-u_h\right>\label{eq-analysis-0}
\end{align}
Both the first and the second term in (\ref{eq-analysis-0}) can be estimated along two paths.
For the first term, (\ref{eq:ineq5}) gives the following estimation
\begin{align}
  \label{eq:analysis-1}
  \left<L_\delta e_h, u-I_h u\right>_\Omega\leq CE_\delta(e_h)\norm{u-I_hu}_{H^1(\Omega)}\leq C\frac{h^{k+1}}{\rho}E_\delta(e_h)\norm{u}_{H^{k+1}(\Omega)}.
\end{align} 
Here the following classical projection error estimation in finite element method 
\begin{align}
  \label{eq:projection error}
  \norm{u-I_h u}_{H^1(\Omega)}\leq C\frac{h^{k+1}}{\rho}\norm{u}_{H^{k+1}(\Omega)}
\end{align}
is applied in the last inequality.

Meanwhile, if we apply (\ref{eq:ineq4}), the first term can be estimated in another way as follows
\begin{align}
  \label{eq:analysis-1-1}
  \left<L_\delta e_h, u-I_h u\right>_\Omega\leq \norm{L_\delta e_h}_{L^2(\Omega)}\norm{u-I_h u}_{L^2(\Omega)}\leq C\frac{h^{k+1}}{\delta}E_\delta(e_h)\norm{u}_{H^{k+1}(\Omega)}.
\end{align}
In the second inequality, we use another projection error estimation
\begin{align}
  \label{eq:projection error2}
  \norm{u-I_h u}_{L^2(\Omega)}\leq Ch^{k+1}\norm{u}_{H^{k+1}(\Omega)}.
\end{align}
The results in (\ref{eq:analysis-1}) and (\ref{eq:analysis-1-1}) can be combined into
\begin{align}
  \label{eq:combine1}
  \left<L_\delta e_h, u-I_h u\right>_\Omega\leq C\frac{h^{k+1}}{\max\{\rho,\delta\}}E_\delta(e_h)\norm{u}_{H^{k+1}(\Omega)}.
\end{align}

We next turn to the second term in (\ref{eq-analysis-0}). On the one hand, this term can be estimated as follows.
\begin{align}
  \left<r_{in}+r_{bd}, I_h u-u_h\right>&\leq \norm{r_{in}}_{L^2(\Omega)}\norm{I_h u-u_h}_{L^2(\Omega)}+C\delta \norm{u}_{H^3(\Omega)}E_\delta(I_hu-u_h)\notag\\
  &\leq C\delta \norm{u}_{H^3(\Omega)}E_\delta(I_hu-u_h)\notag\\
  &\leq C\delta \norm{u}_{H^3(\Omega)}E_\delta (u-I_h u)+C\delta \norm{u}_{H^3(\Omega)}E_\delta (e_h)\notag\\
  &\leq C\delta \norm{u}_{H^3(\Omega)}\norm{u-I_h u}_{H^1(\Omega)}+C\delta \norm{u}_{H^3(\Omega)}E_\delta (e_h)\notag\\
  &\leq C\delta \frac{h^{k+1}}{\rho}\norm{u}_{H^3(\Omega)}\norm{u}_{H^{k+1}(\Omega)}+C\delta \norm{u}_{H^3(\Omega)}E_\delta (e_h).\label{eq:analysis-2}
\end{align}
Here Lemma \ref{lemma:boundary error}, Lemma \ref{lemma: subadditive} and (\ref{eq: ineq3})(\ref{eq:projection error}) are used in the above calculation. 
Additionally, in the fourth line, the estimation $E_\delta(u-I_hu)\leq C\norm{u-I_h u}_{H^1(\Omega)}$ is a natural corollary of Lemma \ref{lemma:Taylor1}.

On the other hand, the second term (\ref{eq-analysis-0}) can be estimated in another way. 
With Lemma \ref{lemma:boundary error}, Lemma \ref{lemma:truncation error}, (\ref{eq: ineq1}) and (\ref{eq:projection error2}),
\begin{align}
  \left<r_{in}+r_{bd}, I_h u-u_h \right>_\Omega
  &\leq  \|r_{in}\|_{L^2(\Omega)}\|I_h u-u_h\|_{L^2(\Omega)}+C\delta \|u\|_{H^3(\Omega)}  E_\delta(I_h u-u_h) \notag\\
  &\leq  C \delta\|u\|_{H^3(\Omega)} E_\delta(I_h u-u_h) \notag\\
  &\leq C \delta\|u\|_{H^3(\Omega)} E_\delta ( u-I_h u)+C \delta\|u\|_{H^3(\Omega)} E_\delta (e_h) \notag\\
  &\leq C \|u\|_{H^3(\Omega)} \|u-I_h u\|_{L^2(\Omega)}+C \delta\|u\|_{H^3(\Omega)} E_\delta (e_h) \notag\\
  &\leq C \delta\frac{h^{k+1}}{\delta}\|u\|_{H^3(\Omega)}\|u\|_{H^{k+1}(\Omega)} +C \delta\|u\|_{H^3(\Omega)} E_\delta (e_h).\label{eq:analysis-2-1}
\end{align}

Combining (\ref{eq:analysis-1})(\ref{eq:analysis-2}), we can get 
\begin{align*}
  (E_\delta(e_h))^2&=\left<L_\delta e_h,e_h\right>_\Omega\\
  &\leq C\frac{h^{k+1}}{\rho}E_\delta(e_h)\norm{u}_{H^{k+1}(\Omega)}+C\delta \frac{h^{k+1}}{\rho}\norm{u}_{H^3(\Omega)}\norm{u}_{H^{k+1}(\Omega)}+C\delta \norm{u}_{H^3(\Omega)}E_\delta (e_h),
\end{align*}
which implies that 
\begin{align*}
  E_\delta(e_h)\leq C\left(\frac{h^{k+1}}{\rho}+\delta\right)\norm{u}_{H^{\max\{k+1,3\}}(\Omega)}.
\end{align*}
In addition, combining (\ref{eq:analysis-1-1})(\ref{eq:analysis-2-1}), we can get 
\begin{align*}
  (E_\delta(e_h))^2&=\left<L_\delta e_h,e_h\right>_\Omega\\
  &\leq C\frac{h^{k+1}}{\delta}E_\delta(e_h)\norm{u}_{H^{k+1}(\Omega)}+C\delta \frac{h^{k+1}}{\delta}\norm{u}_{H^3(\Omega)}\norm{u}_{H^{k+1}(\Omega)}+C\delta \norm{u}_{H^3(\Omega)}E_\delta (e_h),
\end{align*}
which implies another estimation 
\begin{align*}
  E_\delta(e_h)\leq C\left(\frac{h^{k+1}}{\delta}+\delta\right)\norm{u}_{H^{\max\{k+1,3\}}(\Omega)}.
\end{align*}

These two results can have a combined form like (\ref{eq:combine1}). Moreover,
with (\ref{eq: ineq3}), we get a unified $L^2$ error estimation 
\begin{align*}
  \norm{e_h}_{L^2(\Omega)}\leq C\left(\frac{h^{k+1}}{\max\{\rho,\delta\}}+\delta\right)\norm{u}_{H^{\max\{k+1,3\}}(\Omega)}.
\end{align*}
Here the Theorem \ref{th:main} has been proved.
\subsection{Convergent approximation of gradient.}
We can further give an approximation to the gradient of the local solution. As mentioned in Remark \ref{remark:gradient approximation}, 
when we get the finite element solution $u_h$, $\nabla S_\delta u_h$ can serve as an approximation of $\nabla u$. In this section, we mainly focus on the proof of (\ref{eq:gradient approximation1}) and (\ref{eq:gradient approximation2}).
In our proof, the necessity to introduce a correction term in Theorem \ref{th:gradient} can be observed. As for the complicated proof of Theorem \ref{th:gradient}, it can be found in Appendix \ref{appendix:lemma-correction}.

We firstly divide the gradient error into two parts as follows.
\begin{align}
  \label{eq:analysis-3}
  \norm{\nabla u-\nabla S_\delta u_h}_{L^2(\Omega)}\leq C\norm{\nabla u-\nabla S_\delta u}_{L^2(\Omega)}+\norm{\nabla S_\delta e_h}_{L^2(\Omega)}.
\end{align}
The second term in (\ref{eq:analysis-3}) is easy to bound because from Lemma \ref{lemma:three results} we can get 
\begin{align}
  \label{eq:analysis-4}
  \norm{\nabla S_\delta e_h}_{L^2(\Omega)}\leq CE_\delta(e_h)\leq C\left(\frac{h^{k+1}}{\max\{\rho,\delta\}}+\delta\right)\norm{u}_{H^{\max\{k+1,3\}}(\Omega)}.
\end{align}

The remaining first term is independent of the finite element method. 
To estimate this term, we need more calculation. 
\begin{align}
  &\nabla (u-S_\delta u)\notag\\
  =&\nabla u(\xx)-\frac{1}{w_\delta(\xx)}\into \nabla_\xx \Rd u(\yy)\d\yy+\frac{1}{w^2_\delta(\xx)}\into \Rd u(\yy)\d\yy\into \nabla_\xx\Rd\d\yy\notag\\
  =&\nabla u(\xx)+\frac{1}{w_\delta(\xx)}\into \nabla_\yy \Rd u(\yy)\d\yy-\frac{1}{w^2_\delta(\xx)}\into \Rd u(\yy)\d\yy\into \nabla_\zz R_\delta(\xx,\zz)\d\zz\notag\\
  =&\nabla u(\xx)-\frac{1}{w_\delta(\xx)}\into \Rd \nabla u(\yy)\d\yy+\frac{1}{w_\delta(\xx)}\intpo \Rd u(\yy)\nn(\yy)\d S_\yy\notag\\
  &\hspace{2cm} - \frac{1}{w_\delta^2(\xx)}\into \Rd u(\yy)\d\yy\intpo R_\delta(\xx,\zz)\nn(\zz)\d S_\zz\notag\\
  =&\frac{1}{w_\delta(\xx)}\into\Rd (\nabla u(\xx)-\nabla u(\yy))\d\yy\notag\\
  &\hspace{2cm}+ \frac{1}{w^2_\delta(\xx)}\intpo\into \Rd R_\delta(\xx,\zz)(u(\zz)-u(\yy))\nn(\zz)\d\yy\d S_\zz.\label{eq:nabla u-Su}
\end{align}
This result indicates when $\xx\in \Omega\backslash\Omega_{2\delta}=\left\{\xx\in\Omega\big|d(\xx,\partial\Omega)>2\delta\right\}$, the second term in (\ref{eq:nabla u-Su}) vanishes. 
Therefore, it suffices to consider only the first term when proving (\ref{eq:gradient approximation2}). 
From Lemma \ref{lemma:Taylor1}, we can estimate the first term above as 
\begin{align}
  &\norm{\frac{1}{w_\delta(\xx)}\into\Rd (\nabla u(\xx)-\nabla u(\yy))\d\yy}_{L^2(\Omega)}^2\notag\\
  \leq& C\into\left(\into \Rd (\nabla u(\xx)-\nabla u(\yy))\d\yy\right)^2\d\xx\notag\\
  \leq& C\into\into \Rd|\nabla u(\xx)-\nabla u(\yy)|^2\d\xx\d\yy\notag\\
  \leq& C\delta^2\norm{u}_{H^2(\Omega)}^2.\label{eq:nabla-first}
\end{align}
Combining (\ref{eq:analysis-4}) and (\ref{eq:nabla-first}), we can conclude (\ref{eq:gradient approximation2}).
As for the proof of (\ref{eq:gradient approximation1}), the second term in (\ref{eq:nabla u-Su}) should be included. We need two lemmas to estimate this term.
\begin{lemma}
  For the kernel $R_\delta$ defined in Section \ref{sec:Nonlocal diffusion model}, we have the following estimation 
  \begin{align*}
    \into \Rd R_\delta(\xx,\zz)\d\xx\leq C R_{2\sqrt{2}\delta}(\yy,\zz),
  \end{align*}
  where $C$ is a constant independent of $\delta$.
\end{lemma}

We have proved this lemma in \cite{meng2023maximum}. The detailed proof can be found in the appendix of this article.
\begin{lemma}
  \label{lemma:Taylor}
  There exists a constant $C$ independent of $\delta$ such that for $u\in H^2(\Omega)$,
  \begin{align*}
    \intpo\into\Rd (u(\xx)-u(\yy))^2\d\xx\d S_\yy\leq C\delta^2\norm{u}_{H^2(\Omega)}^2.
  \end{align*}
\end{lemma}
The proof of this Lemma is put in Appendix \ref{Appendix:A}.

With these two lemmas, we can estimate the second term in (\ref{eq:nabla u-Su}).
\begin{align}
  &\norm{\frac{1}{w^2_\delta(\xx)}\intpo\into \Rd R_\delta(\xx,\zz)(u(\zz)-u(\yy))\nn(\zz)\d\yy\d S_\zz}_{L^2(\Omega)}^2\notag\\
  \leq&C\into\left(\intpo\into R_\delta(\xx,\zz)\Rd |u(\zz)-u(\yy)|\d\yy\d S_\zz\right)^2\d\xx\notag\\
  \leq&C\into\left(\intpo\into R_\delta(\xx,\zz)\Rd (u(\zz)-u(\yy))^2\d\yy\d S_\zz\right)\left(\intpo\into R_\delta(\xx,\zz)\Rd\d\yy\d S_\zz\right)\d\xx\notag\\
  \leq&\frac{C}{\delta}\intpo\into (u(\zz)-u(\yy))^2 \into R_\delta(\xx,\zz)\Rd \d\xx \d\yy\d S_\zz\notag\\
  \leq&\frac{C}{\delta}\intpo\into R_{2\sqrt{2}\delta}(\yy,\zz)(u(\zz)-u(\yy))^2\d\yy\d S_\zz\notag\\
  \leq&C\delta \norm{u}_{H^2(\Omega)}.\label{eq:nabla-second}
\end{align}
Now (\ref{eq:gradient approximation1}) can be derived from (\ref{eq:analysis-4})(\ref{eq:nabla-first}) and (\ref{eq:nabla-second}).

We can find the $\nabla S_\delta u_h$ can only approximate $\nabla u$ with half order about $\delta$ because $\nabla S_\delta u$ loses accuracy near the boundary.
In other words, this relatively low order is caused by the operator $S_\delta$ rather than our finite element scheme.
To deal with this issue, we introduced a correction term in (\ref{eq:correction}). This term is designed for offsetting the principal error in $\nabla u - \nabla S_\delta u$. In detail,
\begin{lemma}
  \label{lemma:correction}
  For $u\in H^3(\Omega)$, we have 
  \begin{align*}
    \norm{\nabla u - (\nabla S_\delta u - \mathbf{F}_\delta)}_{L^2(\Omega)}\leq C\delta \norm{u}_{H^3(\Omega)}.
  \end{align*}
  Here $C$ is a constant independent of $\delta$.
\end{lemma}

The proof of this lemma is put in appendix \ref{appendix:lemma-correction}. With Lemma \ref{lemma:correction} and (\ref{eq:analysis-4}), we can get 
\begin{align*}
  \norm{\nabla u - (\nabla S_\delta u_h - \mathbf{F}_\delta)}_{L^2(\Omega)}&\leq \norm{\nabla u - (\nabla S_\delta u - \mathbf{F}_\delta)}_{L^2(\Omega)}+\norm{\nabla S_\delta e_h}_{L^2(\Omega)}\\
&\leq C\left(\frac{h^{k+1}}{\max\{\rho,\delta\}}+\delta\right)\norm{u}_{H^{\max\{k+1,3\}}(\Omega)}.
\end{align*}
This result means $(\nabla S_\delta u_h - \mathbf{F}_\delta)$ can approximate $\nabla u$ with a satisfactory accuracy. Additionally, both $\nabla S_\delta u_h$ and $\mathbf{F}_\delta$ can be computed efficiently. The implementation detail can be found in Appendix \ref{Appendix:implementation}. 
\section{Fast Implementation}
\label{sec: Fast Implementation}
In this section, a fast implementation of the nonlocal finite element scheme will be illustrated.

Two constrains in our implementation should be explained at first.
\begin{itemize}
  \item[(a)] (Gaussian kernel) $R(r)=e^{-s^2r}$, $r\in [0,\infty)$;
  \item[(b)] (Rectangular partitionable domain) $\Omega = \bigcup_\alpha T_\alpha$, where each $T_\alpha$ is $n$-dimensional tensor-product domain.
\end{itemize}

We remark that although $R(r)$ does not vanish for $r>1$, its exponential decay property still allows our previous proof to hold because we can adjust $s$ to make $R(r)$ small enough in $[1,+\infty)$.
Meanwhile, we require the region to be partitioned into a Cartesian mesh, and the finite element space is the corresponding piece-wise tensor product polynomial space. 
Of course, not all regions can satisfy such strict conditions. However, we typically apply the nonlocal model in a subregion with singularities, which usually allows us to define the required region ourselves. 
Therefore, our method is universally applicable.

With $R(r)$ defined as above, the kernel functions in our nonlocal diffusion model become 
\begin{align*}
  \Rd &= C_\delta e^{-\frac{s^2}{4\delta^2}\sum_{i=1}^n(x_i-y_i)^2}\\
  \bRd&= C_\delta \frac{1}{s^2}e^{-\frac{s^2}{4\delta^2}\sum_{i=1}^n(x_i-y_i)^2}.
\end{align*}
If the finite dimensional space $S_h$ is spanned by a basis $\{\psi_i\big|i=1,\cdots,N\}$, $u_h$ will be expressed as 
$u_h(\xx)=\sum_{i=1}^N c_i \psi_i(\xx)$. Then, following (\ref{eq:fem}), we can get $u_h$ by solving a linear system $\mathbf{A}\mathbf{c}=\mathbf{f}$, with the $(i,j)$-element of $\mathbf{A}$ being 
\begin{align}
  a_{ij}&=\left<L_\delta\psi_i,\psi_j\right>\notag\\
  &=\int_{\Omega}\left(\frac{1}{\delta^2}\int_{\Omega}\Rd (\psi_i(\xx)-\psi_i(\yy))\d\yy+\into \bRd \psi_i(\yy)\right)\psi_j(\xx)\d\xx\notag\\
  &=\frac{1}{\delta^2}\into\into \Rd \psi_j(\xx)\psi_i(\xx)\d\xx\d\yy - \frac{1}{\delta^2}\into\into \Rd \psi_i(\yy)\psi_j(\xx)\d\xx\d\yy\notag\\
  &\hspace{1cm}+\into\into \bRd \psi_i(\yy)\psi_j(\xx)\d\xx\d\yy,\label{eq:Implementation-1}
\end{align}
and the $j$-th component of right-hand side $\mathbf{f}$ being 
\begin{align}
  f_j&=\left<\bar{f}_\delta, \psi_j\right>\notag\\
  &=\into\into\bRd f(\yy)\psi_j(\xx)\d\yy\d\xx+2\into\intpo \bRd g(\yy)\psi_j(\xx)\d S_\yy\d \xx \label{eq:Implementation-2}
\end{align}
As mentioned in Section \ref{Sec:Introduction},
all the three terms in (\ref{eq:Implementation-1}) are in fact $2n$-dimensional integrals.
And we should calculate a series of this kind of integrals to assemble the stiff matrix. 
Moreover, we should also provide method to deal with the boundary integrals in (\ref{eq:Implementation-2}).
\subsection{Computation of stiff matrix.}
To get the stiff matrix, With these two configurations above, we designed a novel implementation which converts each $2n$-dimensional integral into the computation of $n$ double integrals. Moreover, our method entirely avoids the use of numerical quadrature.

We next illustrate our method in detail using the two-dimensional case as an example. Additionally, the constant coefficient $C_\delta$ in kernel functions is ignored as it can be eliminated from both sides of the equation.
In this case, region $\Omega$ is decomposed by rectangles and finite element space can be chosen as the classical piece-wise bilinear, biquadratic or bicubic polynomial space. 
Recalling the integrals in (\ref{eq:Implementation-1}), 
\begin{align*}
  \into\into \Rd \psi_j(\xx)\psi_i(\yy)\d\xx\d\yy
  =\sum_{T\in\mathcal{T}_h}\sum_{T'\in \mathcal{T}_h}\int_T\int_{T'}\Rd \psi_j(\xx)\psi_i(\yy)\d\xx\d\yy
\end{align*}
Following the construction of classical Lagrange element, basis functions have compact support. 
Therefore, same as the implementation of common finite element method, when $T$ and $T'$ traverse the mesh, 
we can compute the local stiffness matrix for each element, and then assemble these matrices into the global stiffness matrix.
In each rectangle $T\subset \text{supp}(\psi)$, the Lagrange basis functions have the following form
\begin{align*}
  &\psi(x_1,x_2)=C(x_1-\mu_1)(x_2-\mu_2),\quad (\text{bilinear}),\\
  &\psi(x_1,x_2)=C(x_1-\mu_{11})(x_1-\mu_{12})(x_2-\mu_{21})(x_2-\mu_{22}),\quad(\text{biquadratic}),\\
  &\psi(x_1,x_2)=C(x_1-\mu_{11})(x_1-\mu_{12})(x_1-\mu_{13})(x_2-\mu_{21})(x_2-\mu_{22})(x_2-\mu_{23}),\quad(\text{bicubic}).
\end{align*}
We uniformly express these forms as 
\begin{align*}
  \psi_j(x_1,x_2)=p_{j1}(x_1)p_{j2}(x_2),
\end{align*}
then, for $T=[a_1,b_1]\times[a_2,b_2]$ and $T'=[a_1',b_1']\times[a_2',b_2']$,
\begin{align}
  &\int_T\int_{T'}\Rd\psi_j(\xx)\psi_i(\yy)\d\xx\d\yy\notag\\
  =&\int_{a_1}^{b_1}\int_{a_2}^{b_2}\int_{a_1'}^{b_1'}\int_{a_2'}^{b_2'}e^{-\frac{s^2}{4\delta^2}[(x_1-y_1)^2+(x_2-y_2)^2]}p_{j1}(x_1)p_{j2}(x_2)p_{i1}(y_1)p_{i2}(y_2)\d x_1\d x_2\d y_1\d y_2\notag\\
  =&\prod_{l=1,2}\left[\int_{a_l}^{b_l}p_{jl}(x_l)\int_{a_l'}^{b_l'}e^{-\frac{s^2}{4\delta^2}(x_l-y_l)^2}p_{il}(y_l)\d y_l\d x_l\right]\label{eq:form1}
\end{align}
Here we can see the 4-fold integral are transferred to the product of two double integrals. 
The third term in (\ref{eq:Implementation-1}) has exactly the same form as above since the two kernel functions differ only by a constant factor. 
Meanwhile, the first term in (\ref{eq:Implementation-1}) can also be treated in a same way with 
\begin{align}
  &\int_T\int_{T'} \Rd \psi_j(\xx)\psi_i(\xx)\d\xx\d\yy
  =\prod_{l=1,2}\left[\int_{a_l}^{b_l}p_{jl}(x_l)p_{il}(x_l)\int_{a_l'}^{b_l'}e^{-\frac{s^2}{4\delta^2}(x_l-y_l)^2}\d y_l\d x_l\right]\label{eq:form2}
\end{align} 

All the double integrals in (\ref{eq:form1}) and (\ref{eq:form2}) can be consolidated into a unified expression
\begin{align}
  \overline{I}(p,q,\lambda,a,b,a',b')=\int_a^b p(x)\int_{a'}^{b'}e^{-\lambda^2 (x-y)^2}q(y)\d y\d x\label{eq:uni-form}
\end{align}
with $p$ and $q$ be one-dimensional polynomials, e.g. $p(x)=p_{jl}(x)p_{il}(x)$, $q(y)=1$ in (\ref{eq:form2}). 

Next, we explain how to compute these double integrals without using numerical quadrature. Some notations should be introduced at first. Let 
\begin{align*}
  \Phi(a,b,\lambda,k)&=\int_a^b x^ke^{-\lambda^2 x^2}\d x\\
  \overline{\Phi}(a,b,l,\lambda,n)&=\int_a^b x^ne^{-\lambda^2(x-l)^2}\d x\\
  I(p,a,b,l,\lambda)&=\int_a^b p(x)e^{-\lambda^2(x-l)^2}\d x,
\end{align*}
where $p(x)$ is a polynomial.
Noticing that 
\begin{align*}
  \Phi(a,b,\lambda,k)&=-\frac{1}{2\lambda^2}x^{k-1}e^{-\lambda^2x^2}\bigg|_a^b+\frac{k-1}{2\lambda^2}\int_a^b x^{k-2}e^{-s^2 x^2}\d x\\
  &=\frac{1}{2\lambda^2}\left(a^{k-1}e^{-\lambda^2 a^2}-b^{k-1}e^{-\lambda^2 b^2}\right)+\frac{k-1}{2\lambda^2}\Phi(a,b,\lambda,k-2),
\end{align*}
we can compute $\Phi(a,b,\lambda,k)$ recursively with the initial two terms 
\begin{align*}
  \int_a^b e^{-s^2x^2}=\frac{\sqrt{\pi}}{2\lambda}(\mathrm{erf}(\lambda b)-\mathrm{erf}(\lambda a)),\quad \int_a^b xe^{-\lambda^2 x^2}\d x=\frac{1}{2\lambda^2}\left(e^{-\lambda^2 a^2}-e^{-\lambda^2 b^2}\right),
\end{align*}
where 
\begin{align*}
  \mathrm{erf}(x)=\frac{2}{\sqrt{\pi}}\int_0^x e^{-t^2}\d t
\end{align*}
is known as Gauss error function, which has already been implemented in many existing scientific computing libraries. 
Furthermore, since 
\begin{align*}
  \overline{\Phi}(a,b,l,\lambda,n)&=\int_a^b(x-l+l)^ne^{-\lambda^2(x-l)^2}\d x\\
  &=\sum_{k=0}^nC_n^kl^{n-k}\int_a^b(x-l)^ke^{-\lambda^2(x-l)^2}\d x\\
  &=\sum_{k=0}^nC_n^kl^{n-k}\Phi(a-l,b-l,\lambda,k)
\end{align*}
the computation of $\overline{\Phi}(a,b,l,\lambda,n)$ can be gained after the implementation of $\Phi(a,b,\lambda,k)$.
Afterwards, the computation of $I(p,a,b,l,\lambda,n)$ becomes straightforward, since
\begin{align*}
  I(p,a,b,l,\lambda)&=\int_a^b\left(\sum_{n=0}^N c_nx^n\right)e^{-\lambda^2(x-l)^2}\d x=\sum_{n=0}^N c_n \overline{\Phi}(a,b,l,\lambda,n).
\end{align*}

Now we can calculate (\ref{eq:uni-form}) with the three functions above. Let $q'(y)$ represent the derivative of $q(y)$ and $P(x)$ denote the antiderivative of $p(x)$, where the associated constant being of no consequence.
Then, using integration by parts, we can get a recursion formula
\begin{align*}
  &\overline{I}(p,q,\lambda, a,b,a',b')\\
  =&P(x)\int_{a'}^{b'}e^{-\lambda^2 (x-y)^2}q(y)\d y\Bigg|_a^b-\int_a^b P(x)\int_{a'}^{b'}\left(\frac{\d}{\d x}e^{-\lambda^2(x-y)^2}\right)q(y)\d y\d x\\
  =&P(x)\int_{a'}^{b'}e^{-\lambda^2 (x-y)^2}q(y)\d y\Bigg|_a^b+\int_a^b P(x)\int_{a'}^{b'}\left(\frac{\d}{\d y}e^{-\lambda^2(x-y)^2}\right)q(y)\d y\d x\\
  =&P(x)\int_{a'}^{b'}e^{-\lambda^2 (x-y)^2}q(y)\d y\Bigg|_a^b+\int_a^b P(x)\left(e^{-\lambda^2(x-y)^2} q(y)\right)\Bigg|_{a'}^{b'}\d x\\
  &\hspace{2cm}-\int_a^b P(x)\int_{a'}^{b'}e^{-\lambda^2 (x-y)^2}q'(y)\d y\d x\\
  =&P(b)I(q,a',b',b,\lambda)-P(a)I(q,a',b',a,\lambda)+q(b')I(P,a,b,b',\lambda)-q(a')I(P,a,b,a',\lambda)\\
  &\hspace{2cm}-\overline{I}(P,q',\lambda, a,b,a',b').
\end{align*}
With each recursion, the degree of the polynomial 
$q$ is reduced by one through derivation, until 
$q$ becomes zero, at which point the last term vanishes. In other words, we can obtain 
$\overline{I}(p,q,\lambda, a,b,a',b')$ by computing multiple instances of $I(p,a,b,l,\lambda)$, and the method for 
calculating $I(p,a,b,l,\lambda)$ has already been provided.
\subsection{Computation of load vector.}
Following the ideas in computing stiff matrix, we can apply similar method to calculate (\ref{eq:Implementation-2}).

The computation of the first term in (\ref{eq:Implementation-2}) can also be reduced to the integrals in rectangles, that is 
\begin{align*}
  \into\into\bRd f(\yy)\psi_j(\xx)\d\yy\d\xx=\sum_{T\in \mathcal{T}_h}\sum_{T'\in \mathcal{T}_h}\int_{T}\int_{T'}\bRd f(\yy)\psi_j(\xx)\d\yy\d\xx.
\end{align*}
In each rectangle $T'$, the Lagrange basis functions are provided. Thus, we adopt a typical practice of replacing $f$ with its Lagrange interpolation in $T'$, i.e. 
\begin{align}
  \int_{T}\int_{T'}f(\yy)\psi_j(\xx)\d\yy\d\xx=\sum_{i\in i(T')}f_i\int_{T}\int_{T'}\bRd \psi_i(\yy)\psi_j(\xx)\d\yy\d\xx,\label{eq:form3}
\end{align}
where the notation $i(T')$ indicates $f$ is interpolated by the basis functions associated with $T'$.
Since we have already described how to compute (\ref{eq:form1}), the computation of 
(\ref{eq:form3}) naturally follows. 

We next turn to the second terms of (\ref{eq:Implementation-2}) in which boundary integral is involved. Based on the precondition that 
$\Omega$ is decomposed by a mesh of rectangles, the boundary of $\Omega$ is naturally assembled by a set of segments. 
If we denote this set as $\mathcal{L}_h$, the second terms of (\ref{eq:Implementation-2}) can be written as 
\begin{align}
  &\into\intpo \bRd g(\yy)\psi_j(\xx)\d S_\yy\d \xx\notag\\
  =&\sum_{T\in\mathcal{T}_h}\sum_{L'\in \mathcal{L}_h}\int_T\int_{L'}\bRd g(\yy)\psi_j(\xx)\d S_\yy\d\xx\notag\\
  =&\sum_{T\in\mathcal{T}_h}\sum_{L'\in \mathcal{L}_h}\sum_{i\in i(L)}g_i\int_T\int_{L'}\bRd \widetilde{\psi}_i(\yy)\psi_j(\xx)\d S_\yy\d\xx.\label{eq:form4}
\end{align}
In (\ref{eq:form4}), we also utilize the trick of interpolation with the only distinction being that it is one-dimensional here. Meanwhile, we use the notation $\widetilde{\psi}_i$ to indicate this difference.

Subsequently, we illustrate how to calculate (\ref{eq:form4}) by taking a horizontal segment $L'=\left\{(y_1,y_2)\big|a'_1\leq y_1\leq b'_1, y_2=l\right\}$ for example. 
\begin{align*}
  &\int_T\int_{L'}\bRd \widetilde{\psi}_i(\yy)\psi_j(\xx)\d S_\yy\d\xx\\
  =&\frac{1}{s^2}\int_{a_1}^{b_1}\int_{a_2}^{b_2}\int_{a_1'}^{b_1'}e^{-\frac{s^2}{4\delta^2}[(x_1-y_1)^2+(x_2-l)^2]}p_{i1}(y_1)p_{j1}(x_1)p_{j2}(x_2)\d y_1\d x_2\d x_1\\
  =&\frac{1}{s^2}\left[\int_{a_1}^{b_1}p_{j1}(x_1)\int_{a_1'}^{b_1'}e^{-\frac{s^2}{4\delta^2}(x_1-y_1)^2}p_{i1}(y_1)\d y_1\d x_1\right]\left[\int_{a_2}^{b_2}e^{-\frac{s^2}{4\delta^2}(x_2-l)^2}p_{j2}(x_2)\d x_2\right]\\
  =&\frac{1}{s^2}\overline{I}(p_{j1},p_{i1},s/(2\delta), a_1,b_1,a_1',b_1')I(p_{j2},a_2,b_2,l,s/(2\delta)).
\end{align*}
At this point, we have solved the computation of load vector.

With stiff matrix and load vector, the finite element solution $u_h$ can be solved. In addition, we also provide the way to approximate $\nabla u$ based on the solution $u_h$ in Section \ref{sec:Error analysis of finite element method}.
The smoothed term $S_\delta u_h$ and correction term $\mathbf{F}_\delta$ can also be calculated with the same framework. The detail for these two terms is put in Appendix \ref{Appendix:implementation}.  

\section{Numerical Experiments}
\label{sec: Numerical Experiments}
In this section, we will exhibit numerical results to validate the error analysis in Section \ref{sec:Error analysis of finite element method} and demonstrate the performance of the proposed numerical method.
All the experiments are conducted in a Macbook Pro (3.2GHz M1 CPU, 16G memory) with code written in C++.
\subsection{Experiments in a 2D rectangular region}
\label{subsec:Experiments in a 2D rectangular region}
The first example is in $\Omega=[0,1]\times[0,1]$. The solution of the local model is set as 
\begin{equation*}
  u(x_1, x_2) = \cos(\pi x_1)\cos(\pi x_2)+x_1x_2,
\end{equation*}
which implies the right-hand side term 
\begin{align*}
  f(x_1,x_2) = -\Delta u+u= (1+2\pi^2)\cos(\pi x_1)\cos(\pi x_2)+x_1x_2
\end{align*}
and the boundary condition
\begin{equation*}
  g(x_1,x_2) = 
  \left\{
    \begin{aligned}
    \pi\cos(\pi x_1)\sin(\pi x_2)-x_1, \qquad x_1\in[0,1],x_2=0,\\
    -\pi\cos(\pi x_1)\sin(\pi x_2)+x_1,\qquad x_1\in[0,1],x_2=1,\\
    \pi\sin(\pi x_1)\cos(\pi x_2)-x_2, \qquad x_1=0,x_2\in[0,1],\\
    -\pi\sin(\pi x_1)\cos(\pi x_2)+x_2, \qquad x_1=0,x_2\in[0,1].
    \end{aligned}
  \right.
\end{equation*}

\subsubsection{Error between $u$ and $u_h$}

The domain is discretized using a uniform $N \times N$ grid, and tensor-product Lagrange basis functions of order $k = 1$ or $2$ are employed. For convergence in $h$, $\delta$ is fixed at $0.001$ while $N$ is progressively doubled. For convergence in $\delta$, $h$ is fixed at $\frac{1}{128}$ (for $k=1$) or $\frac{1}{64}$ (for $k=2$), and $\delta$ is decreased from $\delta_0 = 0.04$

The $L^2$ and $H^1$ errors with different $h$ are shown in Figure \ref{fig:u-uh-h} when $\delta = 0.001$. For $L^2$ error, the convergence rates align with the bound $O(h^{k+1})$ which is one order higher than the theoretical result. This can be attributed to the absence of Aubin-Nitsche Lemma in the nonlocal context. Similar issue is also reported in 
\cite{du2024error}. When $h\le 1/64$ (first order element) or $h\le 1/16$ (second order element), the $L^2$ error plateaus since $\delta$ dominates.
For $H^1$ error, the observed rates are $O(h^k)$, though the analysis does not guarantee this due to the lack of $H^1$ coercivity in the nonlocal model.

Figure \ref{fig:u-uh-delta} demonstrates the convergence in $\delta$, confirming the first-order dependence on $\delta$ as predicted. As $\delta$ decreasing, the error also plateaus when $h$ dominates. Notably, second order elements exhibit lower plateau since high-order element achieves higher accuracy.



\begin{figure}[h]
  \centering
  \begin{subfigure}[b]{0.48\textwidth} 
      \centering
      \includegraphics[width=\textwidth]{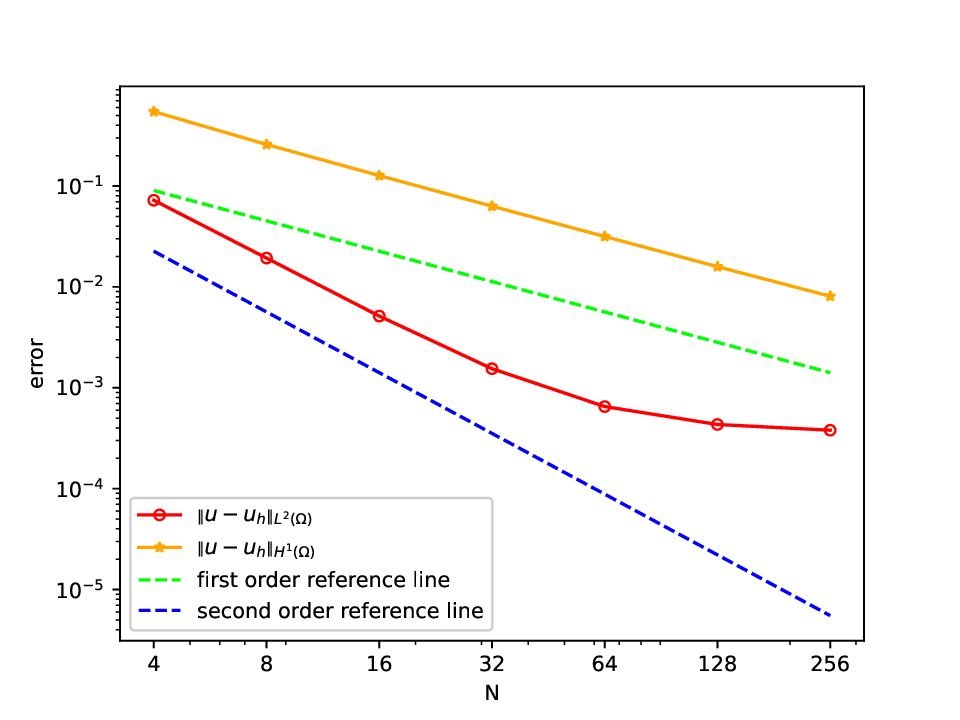} 
      \caption{First order Lagrange interpolation} 
      \label{fig:u-uh-h-left}
  \end{subfigure}
  \hfill 
  \begin{subfigure}[b]{0.48\textwidth} 
      \centering
      \includegraphics[width=\textwidth]{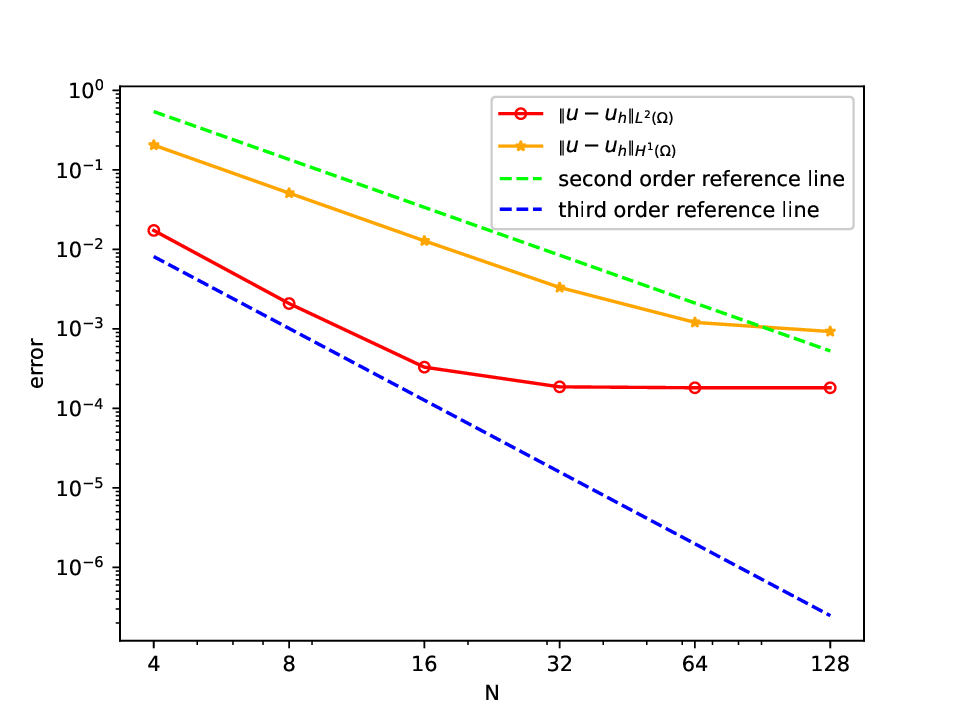} 
      \caption{Second order Lagrange interpolation} 
      \label{fig:u-uh-h-right}
  \end{subfigure}
  \caption{Error between $u$ and $u_h$ with $\delta=0.001$, $h=1/N$.} 
  \label{fig:u-uh-h}
\end{figure}
\begin{figure}[h]
  \centering
  \begin{subfigure}[b]{0.48\textwidth} 
      \centering
      \includegraphics[width=\textwidth]{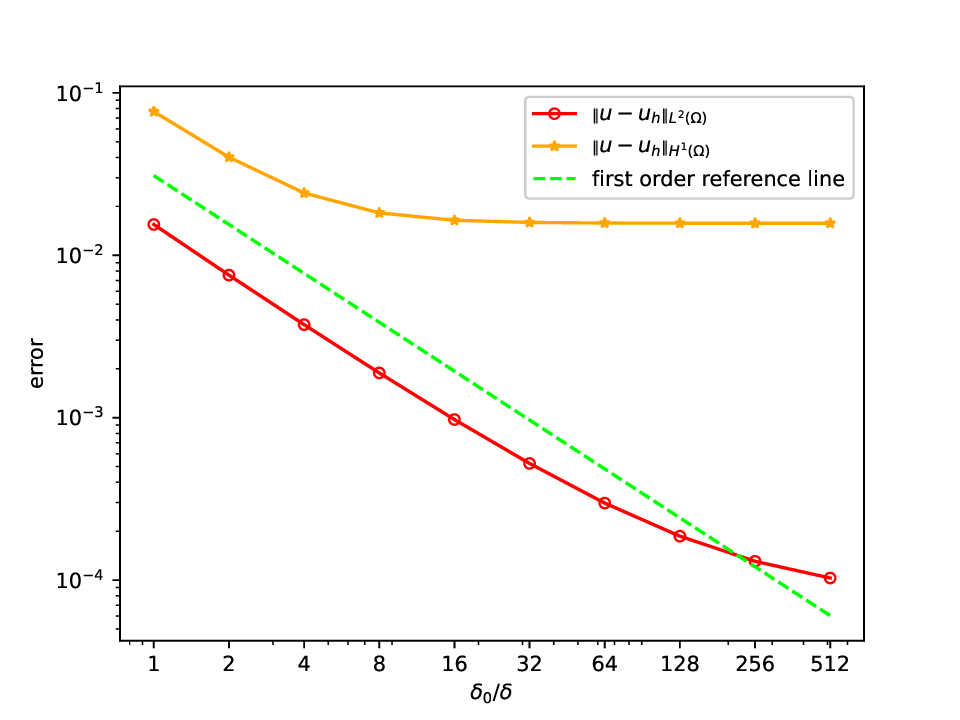} 
      \caption{First order Lagrange interpolation,\\ $N=128$, $\delta_0=0.04$} 
      \label{fig:u-uh-delta-left}
  \end{subfigure}
  \hfill 
  \begin{subfigure}[b]{0.48\textwidth} 
      \centering
      \includegraphics[width=\textwidth]{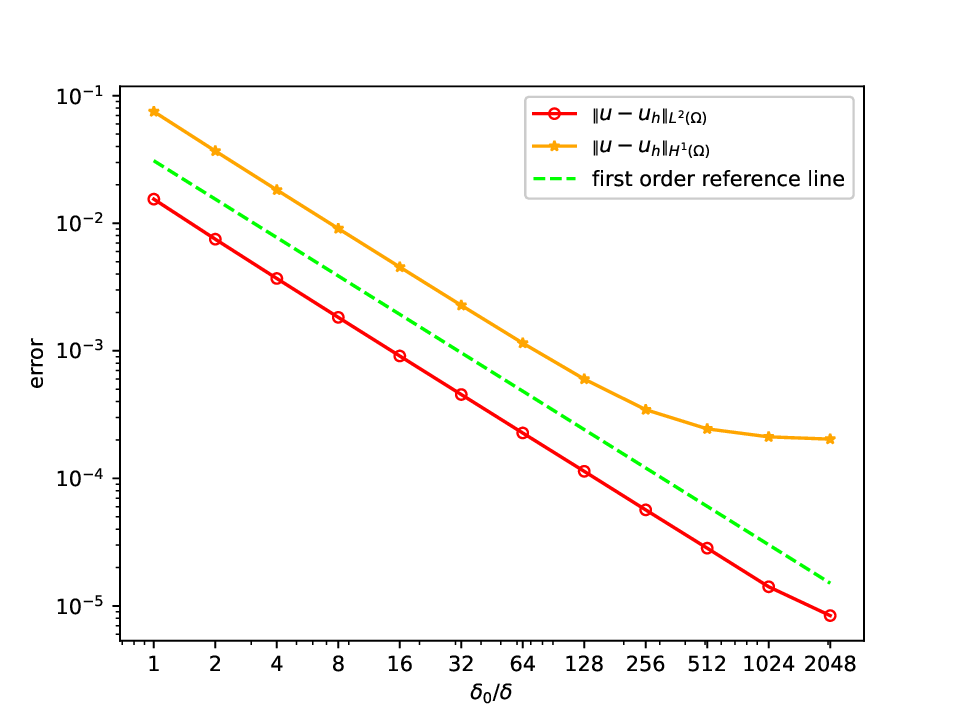} 
      \caption{Second order Lagrange interpolation, $N=64$, $\delta_0=0.04$} 
      \label{fig:u-uh-delta-right}
  \end{subfigure}
  \caption{Error between $u$ and $u_h$ : $h$ is fixed, $\delta$ decreases.} 
  \label{fig:u-uh-delta}
\end{figure}

\subsubsection{Errors in the gradient}
In the theoretical analysis, we propose a gradient recovery technique to approximate the gradient. More precisely, we introduced $S_\delta u_h$ and $\mathbf{F}_\delta$ in Section \ref{sec:Nonlocal diffusion model and conformal finite element discretization} to approximate $\nabla u$. In Table \ref{tab:u-Suh-h} and Table \ref{tab:u-Suh-delta}, we check the convergence of the gradient with respect to $h$ and $\delta$ respectively. As shown in Table \ref{tab:u-Suh-h}, our method gives $k$-th order convergence with respect to $h$. With respect to $\delta$, the convergence rate is first order while the rate is reduced to half order without boundary correction term $\mathbf{F}_\delta$. All the numerical results fit the theoretical analysis very well.

\begin{table}[h]
  \centering
  \begin{tabular}{c|c|c|c|c|c|c|c}
      \hline
      $N$ & 4 & 8 & 16 & 32 & 64 & 128 & 256\\
      \hline
      1st order & 5.42e-1 & 2.58e-1 & 1.27e-1 & 6.31e-2&3.07e-2 &1.30e-2 & 5.54e-3\\
      Rate & - & 1.07 & 1.02 & 1.01 & 1.23 & 1.23 & 1.39 \\
      \hline
      2nd order & 5.10e-2 & 1.28e-2 & 3.46e-3 & 1.55e-3& 1.34e-3& 1.33e-3& 1.33e-3\\
      Rate & - & 1.99  &  1.89 &  1.15 & 0.20 & 0.01  & 0.00 \\
      \hline
  \end{tabular}
  \caption{$\norm{\nabla u-(\nabla S_\delta u_h-\mathbf{F}_\delta)}_{L^2(\Omega)}$ with first and second order Lagrange Elements: $\delta=0.001$, $h=1/N$.} 
  \label{tab:u-Suh-h} 
\end{table}

\begin{table}[h]
  \centering
  \begin{tabular}{c|c|c|c|c|c|c}
      \hline
      $\delta_0/\delta$ &1 &2 & 4 & 8 & 16 & 32 \\
      \hline
      $\norm{\nabla u-\nabla S_\delta u_h}_{L^2(\Omega)}$ & 9.97e-2 & 6.28e-2 & 4.19e-2 & 2.88e-2& 1.99e-2 & 1.49e-2 \\
      Rate & - & 0.67 & 0.58 & 0.54 & 0.53& 0.42 \\
      \hline
      $\norm{\nabla u-\nabla S_\delta u_h}_{L^2(\Omega\backslash\Omega_{2\delta})}$ & 4.13e-2 & 2.30e-2 & 1.23e-2 & 6.24e-3 & 3.19e-3& 1.59e-3\\
      Rate & - & 0.84 & 0.90 & 0.98 & 0.97& 1.01 \\
      \hline
     \footnotesize $\norm{\nabla u-(\nabla S_\delta u_h-\mathbf{F}_\delta)}_{L^2(\Omega)}$ & 5.93e-2 & 2.81e-2 &1.37e-2&6.74e-3 &3.35e-3&1.67e-3\\
      Rate & - & 1.08 & 1.04 & 1.02 & 1.01& 1.01 \\
      \hline
  \end{tabular}
  \caption{Gradient approximation with second order Lagrange Elements: $\delta_0=0.04$, $N=128$.} 
  \label{tab:u-Suh-delta} 
\end{table}

\subsubsection{CPU time of constructing stiff matrix}
Beyond convergence rate validation, we quantitatively analyzed the time consumption of stiffness matrix construction. 


In Table \ref{tab:time-h}, we give the CPU time of assembling the stiffness matrix with different $h$. For uniform rectangular mesh, the stiffness matrix is translation invariant. Using this property, the computational cost can be further reduced as shown in Table \ref{tab:time-h}. If the mesh is non-uniform, the translation-invariant property does not hold. In Table \ref{tab:time-h}, we also list the CPU without using the translation-invariant property. As we can see, the translation-invariant property significantly reduce the computational cost. 
\begin{table}[h]
  \centering
  \begin{tabular}{c|c|c|c|c|c|c}
      \hline
      $N=1/h$ & $4$ & $8$ & 16 & 32 & 64 & 128  \\
      \hline 
      \makecell{1st order} & \makecell{0.004\\ (0.019)} & \makecell{0.009\\ (0.096)} & \makecell{0.025\\ (0.413)} & \makecell{0.078\\ (1.71)} & \makecell{0.695\\ (19.2)} & \makecell{5.306\\ (255)} \\
      \hline
      \makecell{2nd order} &\makecell{0.018\\(0.170)} &\makecell{0.052\\(0.889)} & \makecell{0.136\\(3.898)} & \makecell{0.393\\(16.04)}&  \makecell{3.14\\(184)} & \makecell{21.85\\(2464)} \\
      \hline 
  \end{tabular}
  \caption{CPU time (in seconds) of stiffness matrix assembling with $\delta=0.01$. CPU time without using translation invariance is in the brackets.} 
  \label{tab:time-h} 
\end{table}


\begin{table}[h]
  \centering
  \begin{tabular}{c|c|c|c|c}
      \hline
      $\delta$ & $1/100$ & $1/200$ & $1/400$ & $1/800$  \\
      \hline
      \makecell{1st order} &\makecell{5.31\\(255)} &\makecell{2.54\\(78.9)} & \makecell{0.939\\(28.2)} & \makecell{0.925\\(28.1)}  \\
      \hline
      \makecell{2nd order} &\makecell{21.8\\(2464)}& \makecell{10.4\\(748)} & \makecell{4.15\\(266)} & \makecell{4.15\\(266)} \\
      \hline
  \end{tabular}
  \caption{CPU time (in seconds) of stiff matrix assembling with $h=1/128$. CPU time without using translation invariance is in the brackets.} 
  \label{tab:time-delta} 
\end{table}

The CPU time with different $\delta$ is shown in Table \ref{tab:time-delta}. The computational time increases as $\delta$ grows, since more rectangles involve in the computing of local stiffness matrix.

\subsection{Experiments in a 2D L-shaped Region}
To demonstrate the flexibility of our method for non-rectangular domains, we conduct experiments on an L-shaped region composed of two rectangular subdomains:
\begin{equation}
\Omega = [0,1]\times[0,0.5] \cup [0,0.5]\times[0.5,1].\nonumber
\end{equation}
We consider the exact solution:
\begin{equation}
u(x_1,x_2) = x_1\sin(\pi x_2) + x_2\sin(\pi x_1),\nonumber
\end{equation}
which yields the source term:
\begin{equation}
f(x_1,x_2) = \pi^2(x_1\sin(\pi x_2) + x_2\sin(\pi x_1)) + u(x_1,x_2)\nonumber
\end{equation}
and corresponding Neumann boundary conditions. Figure \ref{fig:Configuration L-shaped Region} illustrates both the domain geometry and solution profile. The domain is discretized using a uniform Cartesian grid with special attention to the reentrant corner. The mesh ensures node alignment at $(0.5,0.5)$ by requiring even divisions in both directions. 

\begin{figure}[h]
  \centering
  \begin{subfigure}[b]{0.48\textwidth} 
      \centering
      \begin{tikzpicture}[scale=3.8,baseline=-1.1 cm] 

    \draw[thick] (0,0) -- (1,0) -- (1,0.5) -- (0.5,0.5) -- (0.5,1) -- (0,1) -- cycle;
    \foreach \x in {0.1,0.2,...,0.4} {
        \draw[] (\x,0) -- (\x,1);
    }
    \foreach \x in {0.5,0.6,...,1.0} {
        \draw[] (\x,0) -- (\x,0.5);
    }
    \foreach \x in {0.5,0.6,...,1.0} {
        \draw[] (0,\x) -- (0.5,\x);
    }
    \foreach \x in {0.1,0.2,...,0.4} {
        \draw[] (0, \x) -- (1, \x);
    }
    \node[below] at (1,0) {$1$};
    \node[left] at (0,1) {$1$};
    \node[left] at (0,0.5) {$0.5$};
    \node[below] at (0.5,0) {$0.5$};
    \node[below left] at (0,0) {$0$};

      \end{tikzpicture}
      \caption{Mesh in L-shaped Region}
      \label{subfig:Mesh in L-shaped Region}
  \end{subfigure}
  \hfill 
  \begin{subfigure}[b]{0.48\textwidth}
      \centering
      \includegraphics[width=\textwidth]{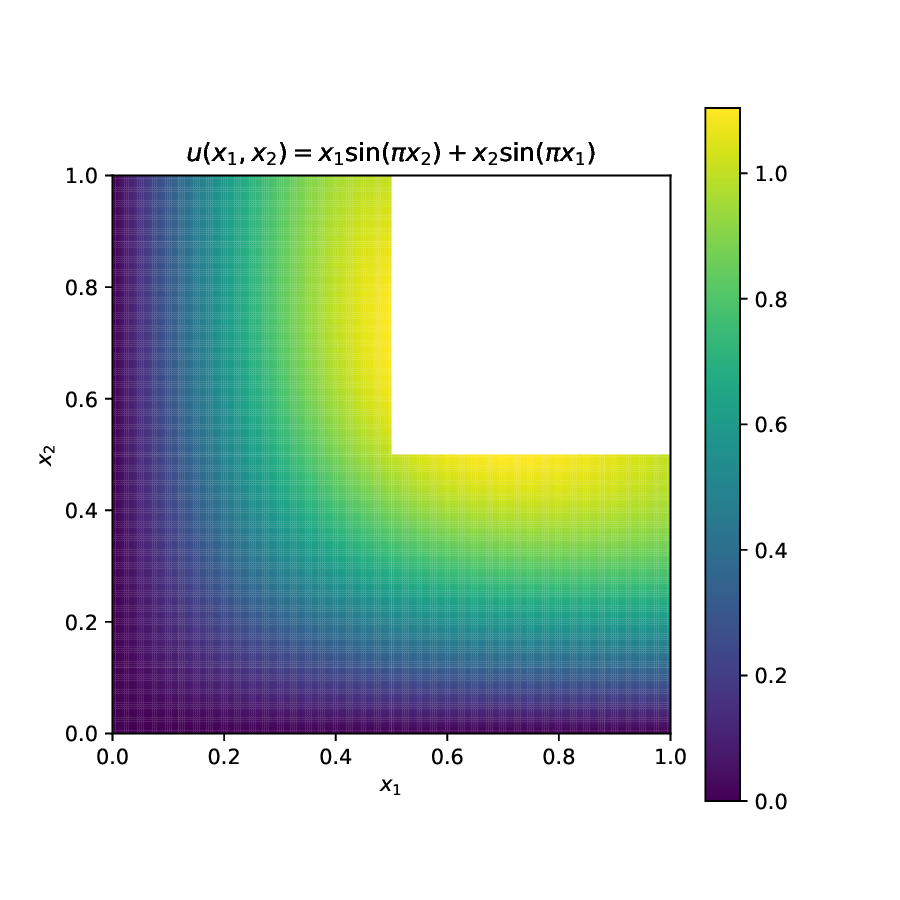} 
      \caption{Exact Solution in L-shaped Region} 
      \label{subfig:exact solution}
  \end{subfigure}
  \caption{Configuration in L-shaped Region} 
  \label{fig:Configuration L-shaped Region}
\end{figure}

We study both the $L^2$ and $H^1$ errors between $u_h$ and $u$. The results are shown in Figure \ref{fig:L-shaped}. For L-shape region, the error has similar behavior as that in 2D box. More precisely, the numerical results verify that $\norm{u-u_h}_{L^2(\Omega)}=O(h^{k+1}+\delta)$ and $\norm{u-u_h}_{H^1(\Omega)}=O(h^k + \delta)$.
\begin{figure}[h]
  \centering
  \begin{subfigure}{0.32\textwidth}
    \centering
    \includegraphics[width=\linewidth]{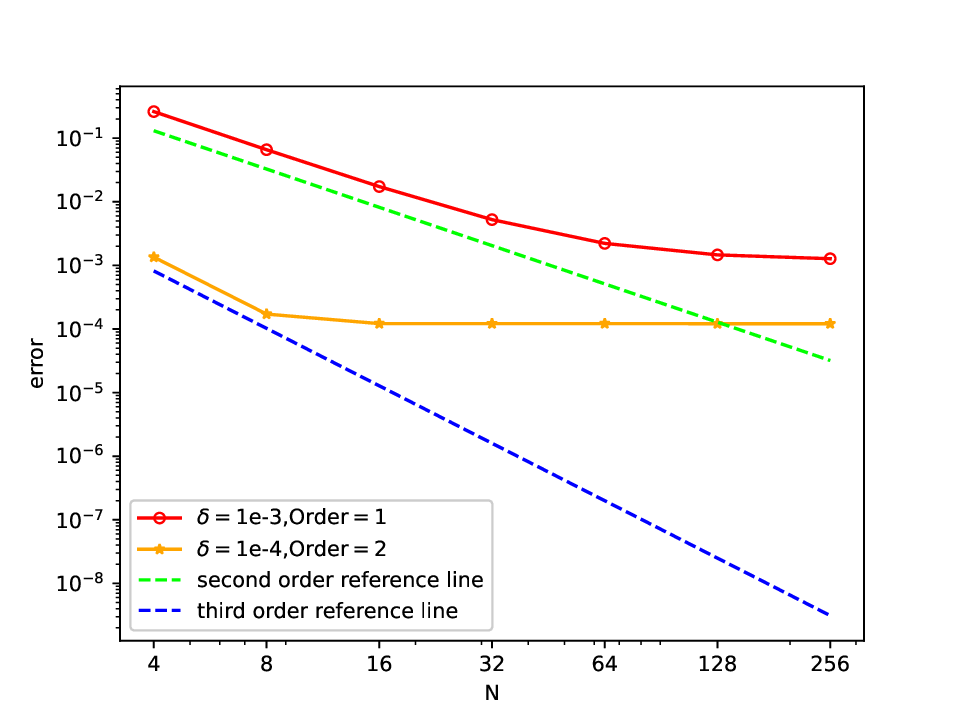}
    \caption{$L^2$ error with $h$ }
    \label{subfig:L-L2}
  \end{subfigure}
  \hfill
  \begin{subfigure}{0.32\textwidth}
    \centering
    \includegraphics[width=\linewidth]{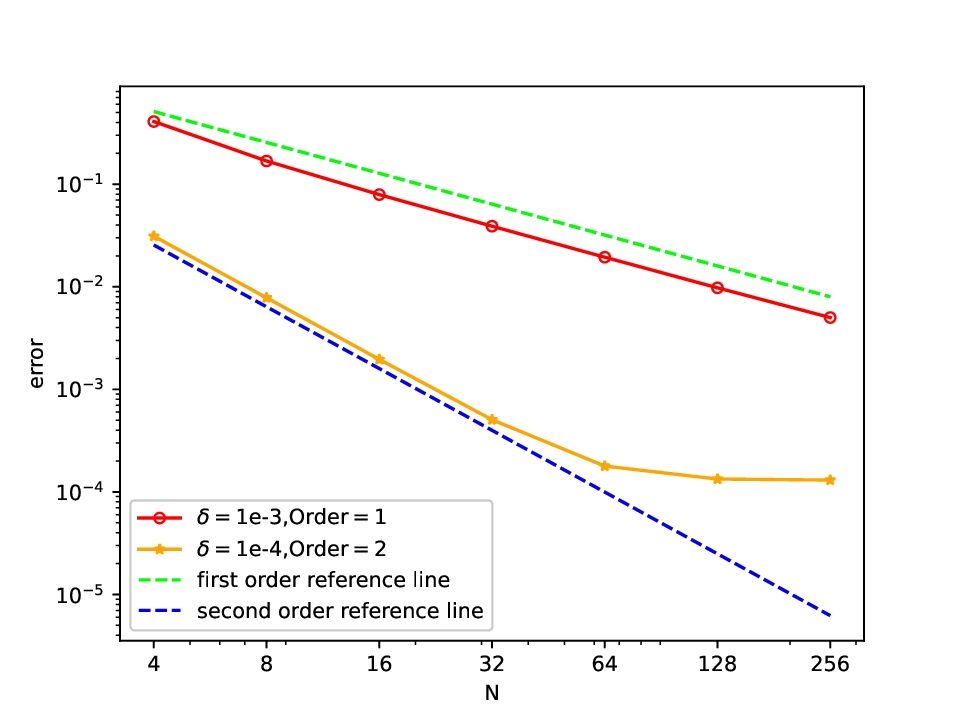}
    \caption{$H^1$ error with $h$ }
    \label{subfig:L-H1}
  \end{subfigure}
  \hfill
  \begin{subfigure}{0.32\textwidth}
    \centering
    \includegraphics[width=\linewidth]{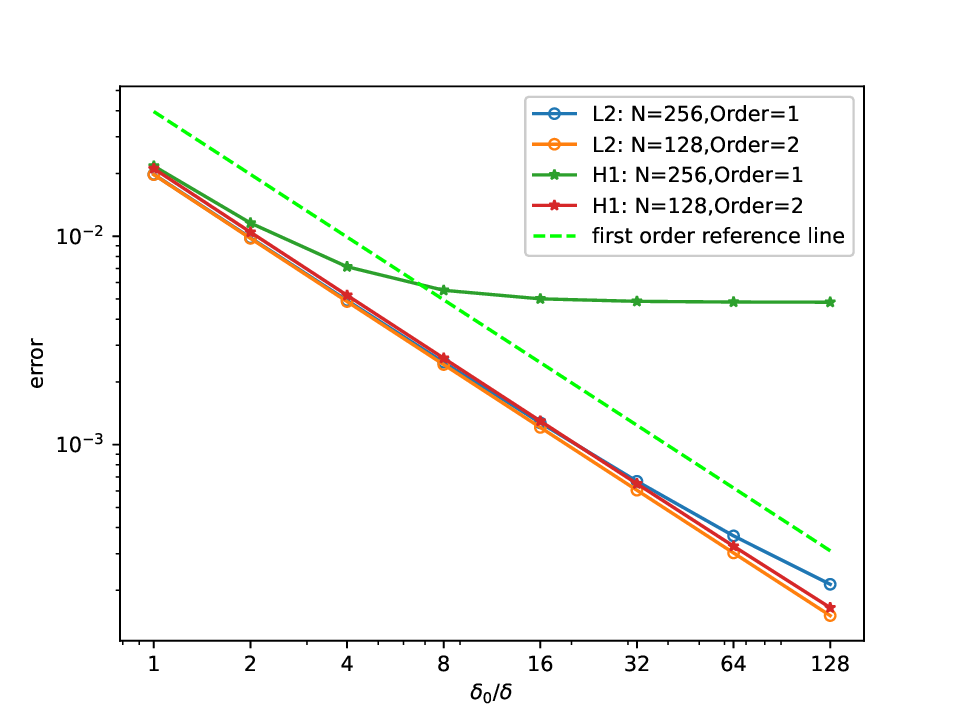}
    \caption{$L^2$ and $H^1$ error with $\delta$}
    \label{subfig:L-delta}
  \end{subfigure}
  \caption{Error in $L$-shaped Region}
  \label{fig:L-shaped}
\end{figure}


\subsection{Experiments in a 3D Cube}

To demonstrate the applicability of our method in higher dimensions, we conduct numerical experiment on a 3D unit cube domain $\Omega = [0,1] \times [0,1] \times [0,1]$. The exact solution is chosen as:
\begin{equation*}
    u(x_1, x_2, x_3) = \cos(\pi x_1)\cos(\pi x_2)\cos(\pi x_3),
\end{equation*}
which yields the source term:
\begin{equation*}
    f(x_1, x_2, x_3) = (1 + 3\pi^2)\cos(\pi x_1)\cos(\pi x_2)\cos(\pi x_3)
\end{equation*}
and homogeneous Neumann boundary conditions $g(x_1, x_2, x_3) = 0$. The cubic domain is discretized using uniform partitions with $N$ subdivisions along each coordinate direction. Our implementation extends naturally from the 2D case.

Figure \ref{fig:3d-h} shows the convergence behavior with fixed $\delta = 10^{-4}$ and increasing mesh resolution. The observed convergence rate is $O(h^{k+1})$ in $L^2$ norm and $O(h^k)$ in $H^1$ norm. 
$\delta$-convergence is studied in Figure \ref{fig:3d-delta}).
The results confirm first-order convergence with respect to $\delta$, with the higher-order method showing more pronounced convergence before reaching the discretization error floor.


\begin{figure}[h]
  \centering
  \begin{subfigure}[b]{0.48\textwidth} 
      \centering
      \includegraphics[width=\textwidth]{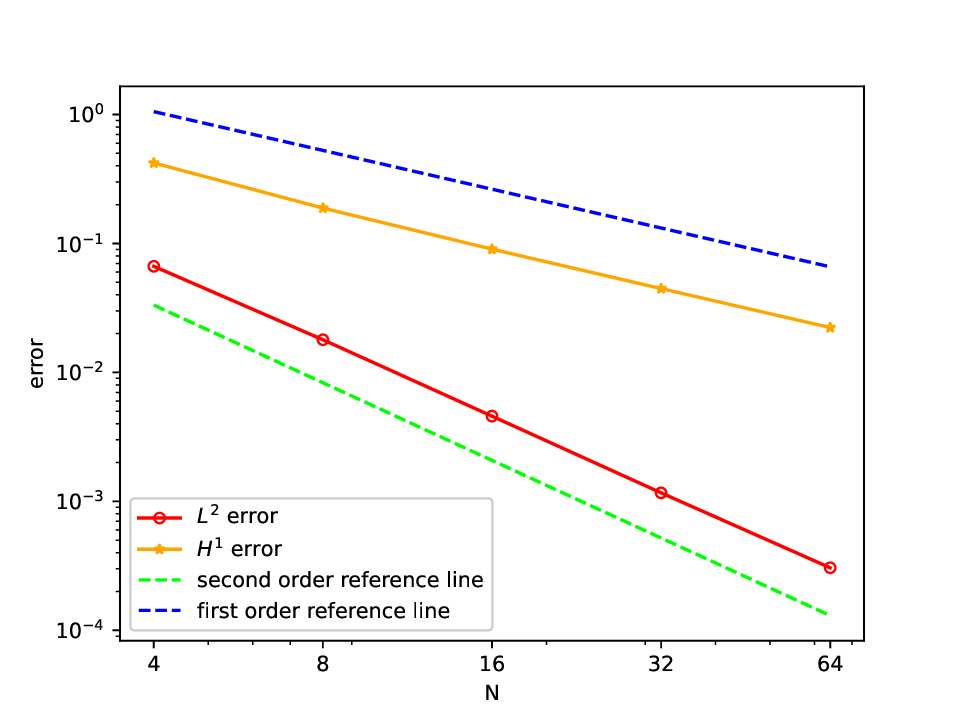} 
      \caption{First order Lagrange element} 
      \label{fig:3d-h-1}
  \end{subfigure}
  \hfill 
  \begin{subfigure}[b]{0.48\textwidth} 
      \centering
      \includegraphics[width=\textwidth]{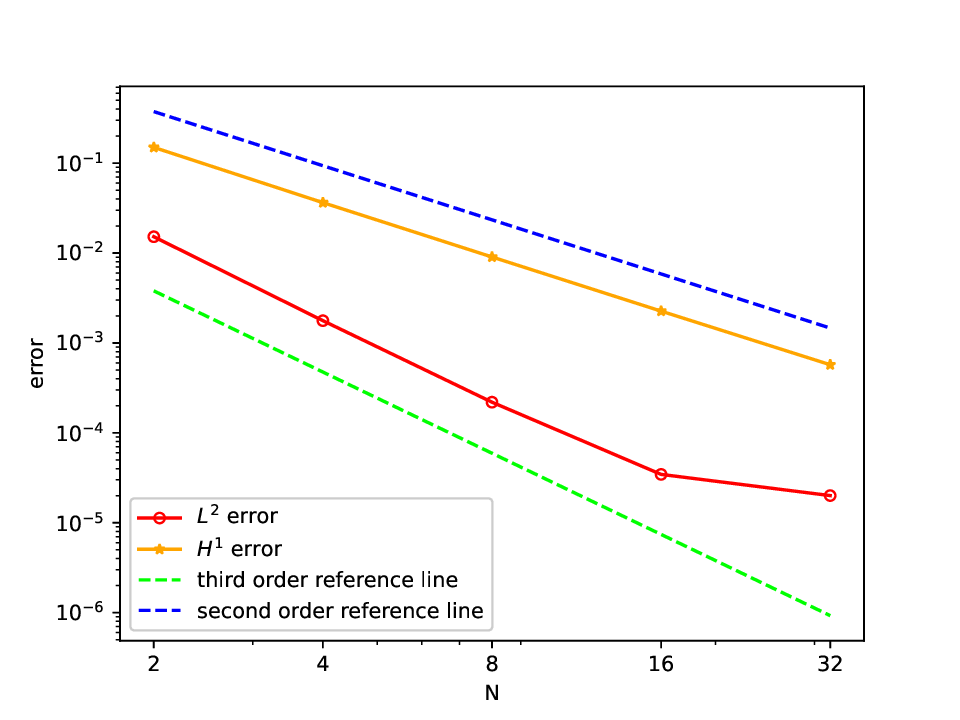} 
      \caption{Second order Lagrange element} 
      \label{fig:3d-h-2}
  \end{subfigure}
  \caption{Error between $u$ and $u_h$ : $\delta=10^{-4}$, $h=1/N$.} 
  \label{fig:3d-h}
\end{figure}
\begin{figure}[h]
  \centering
  \begin{subfigure}[b]{0.48\textwidth} 
      \centering
      \includegraphics[width=\textwidth]{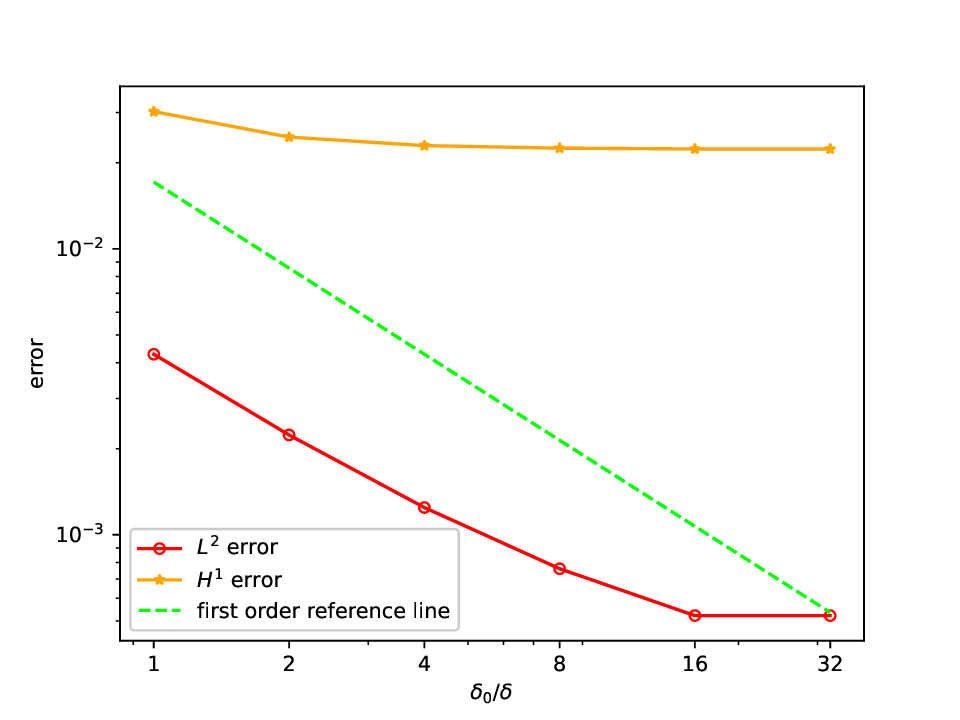} 
      \caption{First order Lagrange element} 
      \label{fig:3d-delta-1}
  \end{subfigure}
  \hfill 
  \begin{subfigure}[b]{0.48\textwidth} 
      \centering
      \includegraphics[width=\textwidth]{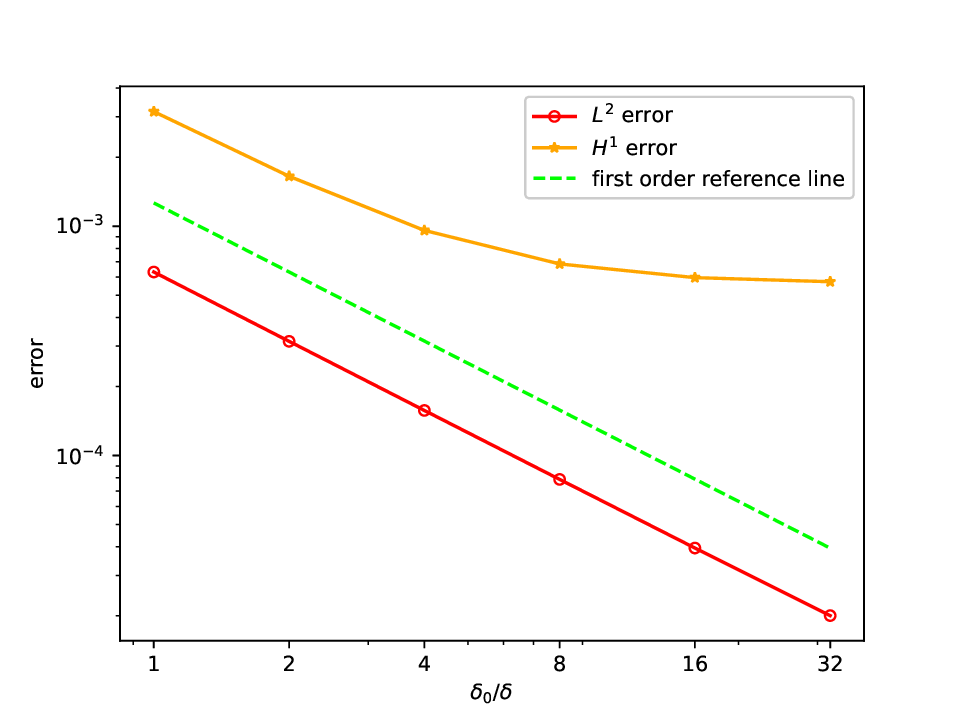} 
      \caption{Second order Lagrange element} 
      \label{fig:3d-delta-2}
  \end{subfigure}
  \caption{Error between $u$ and $u_h$: $h$ is fixed to be $\frac{1}{64}$ and $\frac{1}{32}$, $\delta_0=0.02$ or $0.0016$.} 
  \label{fig:3d-delta}
\end{figure}

\begin{table}[h]
  \centering
  \begin{tabular}{c|c|c|c|c|c}
      \hline
      $N=1/h$ & $4$ & $8$ & $16$ & $32$ & $64$ \\
      \hline 
      \makecell{1st order \\ ($\delta = 0.01$)} & 0.017 & 0.174 & 1.58 & 12.9 & 476 \\
      \hline
      \hline 
      $\delta$ &1/50 & 1/100 & 1/200 & 1/400 & 1/800 \\
      \hline
      \makecell{1st order \\ ($h = \frac{1}{64}$)} & 1246 & 476 & 116 & 116 & 116\\
      \hline 
  \end{tabular}
  \caption{CPU time (in seconds) of constructing stiff matrix with first order element in 3D.} 
  \label{tab:time-order1-3d} 
\end{table}

\begin{table}[h]
  \centering
  \begin{tabular}{c|c|c|c|c|c}
      \hline
      $N=1/h$ & $2$ & $4$ & $8$ & $16$ & $32$ \\
      \hline 
      \makecell{2nd order \\ ($\delta = 0.01$)} & 0.010 & 0.156 & 1.708 & 15.2 & 128\\
      \hline 
      \hline
      $\delta$ &0.02 & 0.01 & 0.005 & 0.0025 & 0.00125 \\
      \hline
      \makecell{2nd order \\ ($h = \frac{1}{32}$)} & 550 & 128 & 128 & 129 & 128\\
      \hline 
  \end{tabular}
  \caption{CPU time (in seconds) of constructing stiff matrix with second order element in 3D.} 
  \label{tab:time-order2-3d} 
\end{table}

We also report the CPU time for constructing stiffness matrix in 3D case, see Table \ref{tab:time-order1-3d} and Table \ref{tab:time-order2-3d}. 
Even for the most expensive case ($N=64,\delta=0.02$), the computation can be done within 1246s in a Macbook laptop which demonstrate the efficiency of the proposed method. 

\section{Conclusion}
\label{sec:conclusion}

This paper has presented a comprehensive framework for finite element approximation of nonlocal diffusion problems, with theoretical analysis and efficient numerical implementation. We proved that the finite element method converges to the correct local limit as both the mesh size $h$ and nonlocal horizon $\delta$ approach zero, without restrictive conditions on their relative scaling. The error analysis establishes $O(h^k + \delta)$ convergence in $L^2$ norm for shape-regular meshes using $k$-th order elements.
For problems requiring gradient approximation, we proposed a post-processing technique combining nonlocal smoothing $S_\delta$ with a boundary correction term $\mathbf{F}_\delta$. This approach is proved to achieve $O(h^k + \delta)$ accuracy for the gradient approximation, overcoming the half-order loss near boundaries.
Moreover, for tensor-product domains with Gaussian kernels, we introduced a novel computational strategy that decomposes the $2n$-dimensional integrals into products of 2D integrals. This approach avoids expensive numerical quadrature while maintaining accuracy. The numerical experiments validate our theoretical results and the efficiency of the proposed algorithm across various geometries, including rectangular domains, L-shaped regions, and three-dimensional cubes. 

In the future, we will try to extend this numerical method to general domain and kernels, not restrictive to the Gaussian kernel and tensor-product domain. We will also explore the application in complex problems include multiscale materials, fracture mechanics etc.

\section*{Appendices}
In the following appendices, we will give the proof of Lemma \ref{lemma:Taylor} and Lemma \ref{lemma:correction}. In the configuration of our nonlocal finite element method, the domain $\Omega$ is set to be a polyhedron. 
Here, for brevity of the proof, we omit some geometric details and restrict our discussion to the case where $\Omega$ is a two-dimensional rectangle.
Moreover, the implementation detail for approximating $\nabla u$ is also provided. 
\appendix
\section{Proof of Lemma \ref{lemma:Taylor}}
\label{Appendix:A}
To prove Lemma \ref{lemma:Taylor}, a technical lemma should be introduced firstly.
\begin{lemma}
  For a polyhedral $\Omega$, let $\Omega_\delta= \left\{\xx\big|d(\xx,\partial\Omega)\leq \delta\right\}$, then for $u\in H^1(\Omega)$, we have
  \begin{align*}
    \norm{u}_{L^2(\Omega_\delta)}^2\leq C\delta\norm{u}_{H^1(\Omega)}^2.
  \end{align*}
\end{lemma}
The proof of this lemma does not require any special techniques. Let $\Omega = [a_1, b_1]\times [a_2, b_2]$.
Following the notations in Figure \ref{fig:Omegadelta}, $\Omega_\delta\subset\bigcup_{i=1}^4\Omega_\delta^i$.
\begin{figure}[h]
  \centering
  \begin{tikzpicture}[scale=3]
    \draw (0,0) rectangle (1,1);
    \fill[blue!20, opacity=0.5] (0,0) rectangle (1,0.1);
    \fill[blue!20, opacity=0.5] (0,0) rectangle (0.1,1);
    \fill[blue!20, opacity=0.5] (0,0.9) rectangle (1,1);
    \fill[blue!20, opacity=0.5] (0.9,0) rectangle (1,1);
    \node[font=\small] at (0.5,0.05) {$\Omega_\delta^1$};
    \node[font=\small] at (0.05,0.5) {$\Omega_\delta^2$};
    \node[font=\small] at (0.5,0.95) {$\Omega_\delta^3$};
    \node[font=\small] at (0.95,0.5) {$\Omega_\delta^4$};
\end{tikzpicture}
\caption{$\Omega_\delta\subset\bigcup_{i=1}^4\Omega_\delta^i$}
\label{fig:Omegadelta}
\end{figure}
Take the integral in $\Omega_\delta^1$ for example, we have 
\begin{align*}
  \int_{\Omega_\delta^1} u^2(\xx)\d\xx &= \int_{a_1}^{b_1}\int_0^\delta u^2(x_1,x_2)\d x_2\d x_1\\
  &=\int_{a_1}^{b_1}\int_{0}^{\delta} \left(u(x_1,0)+\int_0^{x_2} \partial_2u(x_1,t)\d t\right)^2\d x_2\d x_1\\
  &\leq C\int_{a_1}^{b_1}\int_0^\delta \int_0^{\delta}|\partial_2u(x_1,t)|^2\d t\d x_2\d x_1+C\int_{a_1}^{b_1}\int_0^\delta u^2(x_1,0)\d x_2\d x_2\\
  &\leq C\delta \norm{\partial_2 u}_{L^2(\Omega_\delta)}^2+C\delta \norm{u}_{L^2(\partial\Omega)}^2\\
  &\leq C\delta\norm{u}_{H^1(\Omega)}^2.
\end{align*}
With this lemma, we can prove Lemma \ref{lemma:Taylor}. Since
\begin{align*}
  u(\xx)-u(\yy) = \int_0^1 \frac{\d}{\d s}u(\yy+s(\xx-\yy))\d s= \int_0^1 \nabla u(\yy+s(\xx-\yy))\cdot (\xx-\yy)\d s,
\end{align*}
we have 
\begin{align*}
  \int_{\partial\Omega}\int_\Omega \Rd (u(\xx)-u(\yy))^2\d \xx\d S_\yy
\leq C\delta^2\int_0^1\int_{\partial\Omega}\int_\Omega R_\delta(\xx,\yy)|\nabla u(\yy+s(\xx-\yy))|^2\d\xx\d S_\yy\d s.
\end{align*}
For $s\in (0,1]$,
\begin{align*}
  &\int_{\partial\Omega}\int_\Omega R_\delta(\xx,\yy)|\nabla u(\yy+s(\xx-\yy))|^2\d\xx\d S_\yy\\
\leq&\int_{\partial\Omega}\int_\Omega C_\delta R\left(\frac{|\zz-\yy|^2}{4s^2\delta^2}\right)|\nabla u(\zz)|^2\frac{1}{s^n}\d\zz\d S_\yy\\
=&C\int_{\partial\Omega}\int_{\Omega_{2s\delta}}R_{s\delta}(\zz,\yy)|\nabla u(\zz)|^2\frac{1}{s^n}\d\zz\d S_\yy\\
\leq&\frac{C}{s\delta}\int_{\Omega_{2s\delta}}|\nabla u(\zz)|^2\d \zz\\
\leq& C\norm{u}_{H^2(\Omega)}^2.
\end{align*}
Here we have proved Lemma \ref{lemma:Taylor}.
\section{Proof of Lemma \ref{lemma:correction}}
\label{appendix:lemma-correction}
In this section, we give the proof of Lemma \ref{lemma:correction}. From (\ref{eq:nabla u-Su}), we can get 
\begin{align*}
  &\nabla u(\xx)-\nabla S_\delta u(\xx) + F_\delta(\xx)\\
  =&\frac{1}{w_\delta(\xx)}\int_{\Omega}R_\delta(\xx,\yy)(\nabla u(\xx)-\nabla u(\yy))\d\yy -\frac{1}{w_\delta^2(\xx)}\int_{\partial\Omega}\int_\Omega\Rd R_\delta(\xx,\zz)Tu(\yy,\zz)n(\zz)\d \yy\d S_\zz
\end{align*}
where 
\begin{align*}
  Tu(\yy,\zz)=u(\yy)-u(\zz)-((\yy-\zz)\cdot \nn(\zz))(\nabla u(\zz)\cdot \nn(\zz)).
\end{align*}
The first term above has been estimated in (\ref{eq:nabla-first}). We will focus on the second term in this section.
When $u\in C^3(\bar{\Omega})$, we can divide $Tu(\yy,\zz) = T_1 u(\yy,\zz) + T_2 u(\yy,\zz)$, where 
\begin{align*}
  T_1 u(\yy,\zz) = \nabla u(\zz)\cdot ((\yy-\zz)-((\yy-\zz)\cdot\nn(\zz))\nn(\zz)),
\end{align*}
and
\begin{align*}
  T_2 u(\yy,\zz) = \int_0^1\int_0^1\sum_{i,j}\partial_{ij}u(\zz+st(\yy-\zz))s(y_i-z_i)(y_j-x_j)\d t\d s.
\end{align*}
With the same trick in Appendix \ref{Appendix:A}, we can prove 
\begin{align}
  \norm{\frac{1}{w_\delta^2(\xx)}\int_{\partial\Omega}\int_\Omega\Rd R_\delta(\xx,\zz)T_2u(\yy,\zz)n(\zz)\d \yy\d S_\zz}_{L^2(\Omega)}^2\leq C\delta^3\norm{u}_{H^3(\Omega)}^2.\label{eq:T2}
\end{align}
To estimate the integral about $T_1u(\yy,\zz)$,
we divide $\Omega$ into two parts. In detail, if for the polyhedron $\Omega$, $\partial\Omega=\bigcup_{i=1}^n E_i$, where $E_i$ is the flat face of the boundary of $\Omega$, we can denote
\[D_1=\left\{\xx\in\Omega_{2\delta}, \text{there exists unique }i,B(\xx,2\delta)\cap \partial\Omega \subset E_i\right\},\]
and $D_2=\Omega_\delta\backslash D_1$.

For a fixed $\xx\in D_1$,$\nn(\zz)$ is in fact a constant vector.
\begin{align*}
	&\intpo \into \Rd R_\delta(\xx,\zz)T_1u(\yy,\zz)\d\yy\d S_\zz\\
	=&\intpo \into \Rd R_\delta(\xx,\zz) (\yy-\zz) \cdot (I-\nn(\zz)\nn(\zz)^T)\nabla u(\zz)\d \yy\d S_\zz\\
	=&\intpo R_\delta(\xx,\zz)(\xx-\zz)\cdot\nabla_\Gamma u(\zz) \d S_\zz\into \Rd\d\yy + \intpo R_\delta(\xx,\zz)\nabla_\Gamma u(\zz) \d S_\zz\cdot \into \Rd (\yy-\xx)\d\yy
\end{align*}
For the second term, symmetry implies $\into \Rd (\yy-\xx)\d\yy$ is parallel to $\nn(\zz)$. Meanwhile, $\nabla_\Gamma u(\zz)$ is orthogonal to $\nn(\zz)$.
Hence, the second term is in fact zero. As for the first term, using integration by parts on $\partial\Omega$, we have 
\begin{align*}
	\intpo R_\delta(\xx,\zz)(\xx-\zz)\cdot\nabla_\Gamma u(\zz) \d S_\zz=\delta^2\int_{\partial\Omega}\bar{R}_\delta(\xx,\zz)\Delta_\Gamma u(\zz)\d\zz.
\end{align*}
Now, we have 
\begin{align}
	&\int_{D_1}\left(\intpo \into \Rd R_\delta(\xx,\zz)T_1u(\yy,\zz)\d\yy\d S_\zz\right)^2\d\xx\notag\\
	=&\int_{D_1}w_\delta^2(\xx)\left(\intpo R_\delta(\xx,\zz)(\xx-\zz)\cdot\nabla_\Gamma u(\zz) \d S_\zz\right)^2\d\xx\notag\\
	\leq& C\delta^4\int_{D_1}\left(\int_{\partial\Omega} \bar{R}_\delta(\xx,\zz)\Delta_\Gamma u(\zz)\d S_\zz\right)\d\xx\notag\\
	\leq& C\delta^3\int_{D_1}\int_{\partial\Omega}\bar{R}_\delta(\xx,\zz)(\Delta_\Gamma u(\zz))^2\d S_\zz\d\xx \notag\\
	\leq& C\delta^3\norm{\Delta_\Gamma u}_{L^2(\partial\Omega)}^2\notag\\
	\leq& C\delta^3\norm{u}_{H^3(\Omega)}^2.\label{eq: T1-1}
\end{align}

\begin{figure}
  \centering
  \begin{tikzpicture}[scale=3]
    \draw (0,0) rectangle (1,1);
    
    \fill[blue!20, opacity=0.5] 
        (0,0) rectangle (0.1,0.1)    
        (0.9,0) rectangle (1,0.1)     
        (0,0.9) rectangle (0.1,1)     
        (0.9,0.9) rectangle (1,1);    
    \draw[dashed] (0.1,0) -- (0.1,0.1) -- (0,0.1); 
    \draw[dashed] (0.9,0) -- (0.9,0.1) -- (1,0.1); 
    \draw[dashed] (0.1,1) -- (0.1,0.9) -- (0,0.9); 
    \draw[dashed] (0.9,1) -- (0.9,0.9) -- (1,0.9);
    \draw[gray, thick, dashed] (0.05,0.05) circle (0.12); 
    \draw[dashed, gray] (0.05,0.17) -- (1.9,0.8);
    \draw[dashed, gray] (0.05,-0.07) -- (1.9,0.0);
    \draw[thin] (2,0) -- (2.8,0);
    \draw[thin] (2,0) -- (2,0.8);
    \draw[dashed] (2.0,0.5) -- (2.5,0.5) -- (2.5,0.0); 
    \fill[blue!20, opacity=0.5]
          (2.0,0) rectangle (2.5,0.5);
    \node[font=\small] at (2.25,0.25) {$D_2^1$};
\end{tikzpicture}
\caption{The Region $D_2$ and $D_2^1$}
\label{fig:D2}
\end{figure}
For $\xx\in D_2$, we just need to discuss each subdomain near the corners of $\Omega$. As shown in Figure \ref{fig:D2}, we firstly take the region $D_2^1$ in the lower-left corner as an example to give the following estimation
\begin{align}
  \int_{D_2^1} u^2(\xx)\d\xx \leq C\delta^2\norm{u}_{H^2(\Omega)}^2, \text{ for } u\in H^2(\Omega).\label{eq:smallregion}
\end{align}
The proof of this inequality is based on the Sobolev embedding. For $u\in H^2(\Omega)$ and the dimension of $\Omega$ to be $2$ or $3$, we have 
\begin{align*}
  \norm{u}_{C(\bar{\Omega})}\leq C\norm{u}_{H^2(\Omega)}.
\end{align*} 
With this inequality, we can get 
\begin{align*}
  \int_{D_2^1} u^2(\xx)\d\xx
  &\leq C\int_0^{2\delta}\int_0^{2\delta}\left[\left(\int_0^{x_2}\partial_2 u(x_1,t)\d t\right)^2+\left(\int_0^{x_1}\partial_1 u(s,0)\d s\right)^2\right]\d x_1\d x_2\\
  &\hspace{2cm}+C\int_0^{2\delta}\int_0^{2\delta} u^2(0,0)\d x_1\d x_2\\
  &\leq C\delta \int_0^{2\delta}\int_0^{2\delta}\int_0^{2\delta}\left(|\partial_2 u(x_1,t)|^2 + |\partial_1 u(t,0)|^2\right)\d t\d x_1\d x_2 +C\delta^2 u^2(0,0)\\
  &\leq C\delta^2\norm{u}_{H^2(\Omega)}^2.
\end{align*}
Now we can obtain 
\begin{align*}
	&\int_{D_2^1}\left(\intpo \into \Rd R_\delta(\xx,\zz)T_1u(\yy,\zz)\d\yy\d S_\zz\right)^2\d\xx\\
	\leq&C\delta^2\int_{D_2^1}\left(\intpo R_\delta(\xx,\zz)\nabla u(\zz) \d S_\zz\right)^2\d\xx\\
	\leq& C\delta\int_{D_2^1}\intpo R_\delta(\xx,\zz)|\nabla u(\zz)-\nabla u(\xx)|^2 \d S_\zz\d\xx+C\int_{D_2^1}|\nabla u(\xx)|^2\d\xx\\
  \leq&C\delta^3\norm{u}_{H^3(\Omega)}^2+C\delta^2\norm{u}_{H^3(\Omega)}^2
\end{align*}
In the last inequality, we used Lemma \ref{lemma:Taylor} and (\ref{eq:smallregion}). Here we can conclude 
\begin{align}
  \int_{D_2}\left(\int_{\partial\Omega}\int_\Omega R_\delta(\xx,\yy)R_\delta(\xx,\zz)T_1 u(\yy,\zz)\d\yy\d S_\zz\right)^2\d \xx\leq C\delta^2\norm{u}_{H^3(\Omega)}^2\label{eq:T1-2}
\end{align}
Combining (\ref{eq:nabla-first})(\ref{eq:T2})(\ref{eq: T1-1}) and (\ref{eq:T1-2}), Lemma \ref{lemma:correction} is finally proved.
\section{The implementation for approximating $\nabla u$}
\label{Appendix:implementation}
\subsection{Computation of $S_\delta u_h$.} After solving $u_h$ from the linear system, we can
take $\nabla S_\delta u_h$ as the solution. Thus, we should also provide the method to calculate $\nabla S_\delta u_h$.

In fact, by only a little modification in the methods above can we get the value of $\nabla S_\delta u_h(\xx)$ for $\xx\in\Omega$. 
Recalling the definition in (\ref{eq:Sdelta}), we just need to compute
\begin{align*}
  \into \Rd\d\yy=\sum_{T'\in \mathcal{T}_h}\int_{T'}\Rd\d\yy,\quad \into \nabla_\xx\Rd\d\yy=\sum_{T'\in \mathcal{T}_h}\int_{T'}\nabla_\xx\Rd\d\yy,
\end{align*}
and 
\begin{align*}
  \into \Rd u_h(\yy)\d\yy&=\sum_{T'\in \mathcal{T}_h}\sum_{i\in i(T')}(u_h)_i\int_{T'}\Rd \psi_i(\yy)\d\yy.\\
  \into \nabla_\xx\Rd u_h(\yy)\d\yy&=\sum_{T'\in \mathcal{T}_h}\sum_{i\in i(T')}(u_h)_i\int_{T'}\nabla_\xx\Rd \psi_i(\yy)\d\yy.
\end{align*}
Noticing 
\begin{align*}
  \int_{T'}\Rd\d\yy&=\left[\int_{a_1'}^{b_1'}e^{-\frac{s^2}{4\delta^2}(x_1-y_1)^2}\d y_1\right]\left[\int_{a_2'}^{b_2'}e^{-\frac{s^2}{4\delta^2}(x_2-y_2)^2}\d y_2\right]\\
  &=\bar{\Phi}(a_1',b_1',x_1,s/(2\delta),0)\bar{\Phi}(a_2',b_2',x_2,s/(2\delta),0),
\end{align*}
and 
\begin{align*}
  &\int_{T'} \frac{\partial}{\partial x_1}\Rd \psi_i(\yy)\d\yy\\
  =&\frac{s^2}{2\delta^2}\left[\int_{a_1'}^{b_1'}e^{-\frac{s^2}{4\delta^2}(x_1-y_1)^2} (y_1-x_1)p_{i1}(y_1)\d y_1\right]\left[\int_{a_2'}^{b_2'}e^{-\frac{s^2}{4\delta^2}(x_2-y_2)^2} p_{i2}(y_2)\d y_2\right]\\
  =&I((y-x_1)p_{i1},a_1',b_1',x_1,s/(2\delta))I(p_{i2},a_2',b_2',x_2,s/(2\delta)),
\end{align*}
can be reduced to the familiar integrals, the computation of $\nabla S_\delta u_h$ has been solved since the remaining necessary components can also be handled with the same way. 
\subsection{Computation of correction term}
The correction term is defined in (\ref{eq:correction}). Here, we illustrate the computation of this term. Similar to the calculation above, we write 
\begin{align*}
  \mathbf{F}_\delta(\xx) &= \sum_{T\in \mathcal{T}_h}\sum_{L'\in \mathcal{L}_h}\frac{1}{w^2_\delta(\xx)}\int_{L'}\int_{T}\Rd R_\delta(\xx,\zz)g(\zz)((\yy-\zz)\cdot \nn(\zz))\nn(\zz)\d S_\zz\d \yy\\
  &=\sum_{T\in \mathcal{T}_h}\sum_{L'\in \mathcal{L}_h}\sum_{i\in i(L')}\frac{g_i}{w^2_\delta(\xx)}\int_{L'}\int_{T}\Rd R_\delta(\xx,\zz)\widetilde{\psi}_i(\zz)((\yy-\zz)\cdot \nn(\zz))\nn(\zz)\d S_\zz\d \yy
\end{align*}
We next take $L'=\left\{(z_1,z_2)\big|a_1'\leq z_1 \leq b_1', z_2 = l\right\}$ and $\nn(\zz)=(0,1)^T,\zz\in L'$ for example to explain our method. In this case, we can just consider the second decomposition
\begin{align*}
  &\int_{L'}\int_{T}\Rd R_\delta(\xx,\zz)g(\zz)((\yy-\zz)\cdot \nn(\zz))n_2(\zz)\d \yy\d S_\zz\\
=&\int_{a_1'}^{b_1'}\int_{a_1}^{b_1}\int_{a_2}^{b_2}e^{-\frac{s^2}{4\delta^2}[(x_1-y_1)^2+(x_2-y_2)^2]}e^{-\frac{s^2}{4\delta^2}[(x_1-z_1)^2+(x_2-l)^2]}(y_2-l)p_{i1}(z_1)\d y_1\d y_2\d z_1\\
=&e^{-\frac{s^2}{4\delta^2}(x_2-l)^2}\int_{a_1'}^{b_2'}e^{-\frac{s^2}{4\delta^2}(x_1-z_1)^2}p_1(z_1)\d z_1\int_{a_1}^{b_1}e^{-\frac{s^2}{4\delta^2}(x_1-y_1)^2}\d y_1\int_{a_2}^{b_2}e^{-\frac{s^2}{4\delta^2}(x_2-y_2)^2}(y_2-l)\d y_2\\
=&e^{-\frac{s^2}{4\delta^2}(x_2-l)^2}I(p_1,a_1',b_1',x_1,s/(2\delta))\Phi(a_1-x_1,b_1-x_1,s/(2\delta),0)I(y-l,a_2,b_2,x_2,s/(2\delta)).
\end{align*}
Here, we once again get the required terms with some elementary integrals.

\bibliographystyle{abbrv}

\bibliography{ref}

\end{document}